\documentclass[11pt,a4paper]{article}
\usepackage{graphicx} 

\usepackage{amsmath,amssymb,amsthm,color,bm,bbm,array,mathtools,mathrsfs}
\usepackage{soul}
\usepackage[left=1in,right=1in,top=1in,bottom=1in]{geometry}
\usepackage[british]{babel}
\usepackage{csquotes}
\usepackage{hyperref} 
\usepackage[shortlabels]{enumitem}
\usepackage{xcolor}
\usepackage{rotating}

\hypersetup{
    colorlinks=false,
    pdfborder={0 0 0},
}

\usepackage[style=ext-numeric-comp,backend=biber,giveninits=true,isbn=false,url=false,doi=false,articlein=false,maxbibnames=9]{biblatex}
\DeclareFieldFormat[article]{title}{#1} 
\DeclareFieldFormat[phdthesis]{title}{#1} 
 \AtEveryBibitem{%
  \clearfield{number}}
  
\usepackage{tikz,pgfplots}
\usetikzlibrary{decorations.markings}
\usetikzlibrary{decorations.pathreplacing}
\tikzset{mynode/.style = {inner = sep = 2pt, fill, outer sep =0, circle}}
\usetikzlibrary{plotmarks}
\usetikzlibrary{patterns}
\usetikzlibrary{calc}
\usetikzlibrary{shapes,arrows}
\usetikzlibrary{fadings}
\makeatletter
\pgfdeclareverticalshading{pgf@lib@fade@north}{100bp}
{color(0bp)=(pgftransparent!0);
 color(5bp)=(pgftransparent!10);
 color(60bp)=(pgftransparent!100); color(80bp)=(pgftransparent!100)}%
\pgfdeclarefading{myfade}{%
  \pgfuseshading{pgf@lib@fade@east}%
}

\usepgfplotslibrary{polar}
\pgfplotsset{ticks=none}
\pgfplotsset{compat=1.18}

\usepackage{titlesec}
\titleformat{\paragraph}[runin]{\normalfont\normalsize\itshape}{\theparagraph}{1em}{}
\titlespacing*{\paragraph}{0pt}{3.25ex plus 1ex minus .2ex}{1em}

\newtheorem{theorem}{Theorem}[section]
\newtheorem{corollary}[theorem]{Corollary}
\newtheorem{proposition}[theorem]{Proposition}
\newtheorem{lemma}[theorem]{Lemma}
\newtheorem{conjecture}[theorem]{Conjecture}

\numberwithin{equation}{section}
\theoremstyle{definition}
\newtheorem{definition}[theorem]{Definition}

\theoremstyle{remark}
\newtheorem{remark}[theorem]{Remark}

\newtheorem*{remark*}{Remark}

\newcommand{\1}[1]{{\mathbbm{1}\mkern -1.5mu}{\{#1\}}}
\newcommand{\ind}{\mathbbm{1}}
\newcommand{\2}[1]{{\mathbbm{1}}_{#1}}

\newcommand{\R}{{\mathbb R}}

\newcommand{\Z}{{\mathbb Z}}
\newcommand{\N}{{\mathbb N}}
\newcommand{\ZP}{{\mathbb Z}_+}
\newcommand{\RP}{{\mathbb R}_+}

\DeclareMathOperator{\Exp}{\mathbb{E}}
\let\Pr\relax
\DeclareMathOperator{\Pr}{\mathbb{P}}

\DeclareMathOperator{\Var}{\mathbb{V}ar}
\DeclareMathOperator{\Cov}{\mathbb{C}ov}

\DeclareMathOperator{\card}{\#}

\newcommand{\Poi}{\mathrm{Po}\,}
\newcommand{\unif}[2]{\mathrm{Unif}\left( {#1}, {#2} \right)}

\newcommand{\eqd}{\overset{\textup{d}}{=}}

\newcommand{\eps}{\varepsilon}

\newcommand{\re}{{\mathrm{e}}}
\newcommand{\ud}{{\mathrm d}}
\newcommand{\rc}{{\mathrm{c}}}

\newcommand{\cB}{{\mathcal B}}

\newcommand{\cC}{{\mathcal C}}
\newcommand{\cCtp}{\cC_{\theta,\phi}}
\newcommand{\cCo}{\cC^\circ}
\newcommand{\cCotp}{\cCo_{\theta,\phi}}
\newcommand{\cD}{{\mathcal D}}
\newcommand{\cDtp}{\cD_{\theta,\phi}}
\newcommand{\cDtpl}{\cD^{(\lambda)}_{\theta,\phi}}

\newcommand{\cL}{{\mathcal L}}
\newcommand{\tcL}{\widetilde{\cL}}
\newcommand{\cLdlf}[1]{\mathop{\cL_{\mathrm{DLF}}}\left( {#1} \right)}

\newcommand{\cN}{{\mathcal N}}
\newcommand{\cP}{{\mathcal P}}
\newcommand{\cQ}{{\mathcal Q}}
\newcommand{\cR}{{\mathcal R}}
\newcommand{\cS}{{\mathcal S}}

\newcommand{\cU}{{\mathcal U}}
\newcommand{\cV}{{\mathcal V}}
\newcommand{\cW}{{\mathcal W}}
\newcommand{\cX}{{\mathcal X}}
\newcommand{\cY}{{\mathcal Y}}

\newcommand{\tJ}{\tilde J}
\newcommand{\bulk}{\mathrm{blk}}
\newcommand{\bndry}{\mathrm{bdy}}
\newcommand{\intm}{\mathrm{itm}}
\newcommand{\cRbulk}{{\cR}_\lambda^{\bulk}}
\newcommand{\cRbndry}{{\cR}_\lambda^{\bndry}}
\newcommand{\cRint}{{\cR}_\lambda^{\intm}}
\newcommand{\cLbulk}{{\cL}_\lambda^{\bulk}}
\newcommand{\cLbndry}{{\cL}_\lambda^{\bndry}}
\newcommand{\cLint}{{\cL}_\lambda^{\intm}}
\newcommand{\tcLbulk}{\widetilde{\cL}_\lambda^{\bulk}}
\newcommand{\tcLbndry}{\widetilde{\cL}_\lambda^{\bndry}}
\newcommand{\tcLint}{\widetilde{\cL}_\lambda^{\intm}}

\newcommand{\0}{\boldsymbol{0}}
\newcommand{\bu}{\boldsymbol{u}}

\newcommand{\bw}{\boldsymbol{w}}
\newcommand{\bx}{\boldsymbol{x}}
\newcommand{\by}{\boldsymbol{y}}
\newcommand{\bz}{\boldsymbol{z}}

\newcommand{\Eo}{E^\mathrm{o}}
\newcommand{\Ea}{E^\mathrm{a}}

\newcommand{\TT}{\mathbb{T}}
\newcommand{\tDelta}{{\widetilde \Delta}}

\newcommand{\toP}{\overset{\mathrm{P}}\to}

\newcommand{\tod}{\overset{\mathrm{d}}\to}
\newcommand{\toL}[1]{\overset{L^{{#1}}}\to}

\DeclarePairedDelimiter\floor{\lfloor}{\rfloor}

\DeclareMathOperator*{\dprime}{\prime \prime}

\newlist{myenumi}{enumerate}{10}
\setlist[myenumi]{leftmargin=0pt, labelindent=\parindent, listparindent=\parindent, labelwidth=0pt, itemindent=!, itemsep=1pt, parsep=4pt}

\newlist{thmenumi}{enumerate}{10}
\setlist[thmenumi]{leftmargin=0pt, labelindent=\parindent, listparindent=\parindent, labelwidth=0pt, itemindent=!}

\title{Gaussian and non-Gaussian limits for the total edge-length of random directional nearest-neighbour graphs}

\author{Nicholas Georgiou\footnote{Department of Mathematical Sciences, 
Durham University, 
United Kingdom.} 
\and 
Qasem M.~Tawhari\footnote{Department of Mathematics,
Faculty of Science,
Jazan University, 
Saudi Arabia.} 
\and 
Andrew Wade\footnotemark[1]}

\begin{document}

\maketitle

\begin{abstract}  
We study planar random spatial graphs defined by inserting an edge between every point
of a homogeneous Poisson process on the unit square and 
a nearest-neighbour within a fixed cone of directions (if there is such a point). We prove large-sample distributional limit theorems for the total Euclidean edge length of the graph, where the limit distribution is either Gaussian (due to contributions from the bulk) or, when one or both of the boundaries of the cone of directions align with the edges of the square, a convolution of Gaussian with a non-Gaussian law arising from boundary effects due to long edges near a boundary.
This work extends previously known cases, including the ordinary nearest-neighbour graph (full cone of directions), where the Gaussian limit theorem
was obtained by Bickel \& Breiman and Avram \& Bertsimas, 
and Bhatt \& Roy's  minimal directed spanning forest 
(with `south-west' directions), where the non-Gaussian limit theorem was obtained by Penrose \& Wade.
\end{abstract}

\medskip


\medskip


\tableofcontents

\section{Introduction}
\label{sec:intro}

\subsection{Overview}
\label{sec:overview}

An established branch of spatial probability is concerned with graphs whose vertices are located at random points distributed in Euclidean space, and whose edges are constructed using a geometrical rule based on proximity of vertices.
Examples of such graphs include the nearest-neighbour graph, 
the optimal Euclidean travelling salesman tour,
the Euclidean minimal spanning tree, 
and various other geometric graphs.
These types of graphs, collectively known as  spatial graphs, 
are in part motivated by problems in applied probability,
since in real-world networks, nodes often carry spatial information relevant
 to a network's structure or properties, for example, in communications,
 transport, and geographical networks. It is natural to study simple models of \emph{random} spatial networks, for
 inference or prediction, for insight into typical behaviour, or for
 assessing performance of algorithms or processes that take place on
 networks. We refer to~\cite{steele-book,yukich-book,penrose-book} for relevant monographs, and~\cite{kesten-lee,bhh,aldous-steele,penrose-yukich-lln,penrose-yukich-clt,bary-yukich,trauthwein,lpy} for a small selection from the large literature.

One of the most fundamental spatial graphs is the nearest-neighbour graph (NNG), which is a directed graph that connects each point to its nearest neighbour. Nearest-neighbour graphs and distances have numerous applications in
 statistics, machine learning, computational geometry, ecology, and so on, for example in classification and clustering. The subject of this paper is
a variant of the ordinary NNG that introduces a restriction on the possible spatial directions of edges. Applications and properties of constrained nearest-neighbour graphs have been studied
over several decades, an early example being Gilbert's ``exodic tree''~\cite{gilbert65} in which edges are directed from a point to its nearest neighbour that is closer to the origin, a model that has recently
been studied in depth under the name \emph{radial spanning tree}~\cite{baccelli-bordenave,schulte-thale}.  Other examples of constrained nearest-neighbour graphs include~\cite{smith,manna2004scale,manna-sen,nandi-manna}.

The specific model we consider 
is the \emph{minimal directed spanning forest} (MDSF), introduced by Bhatt and Roy~\cite{bhatt-roy},
as a potential model for constrained telecommunications or  drainage networks, 
and subsequently studied in~\cite{blp,pw-dickman,pw-south-west,pw-survey,wade-llns,pw-south,bhattacharjee}.
In the version of the MDSF that we study here, the directional constraint is specified by a permitted \emph{cone} of directions, and each point is joined by a (directed) edge to its nearest neighbour restricted to lie within the permitted cone.

We take the underlying (random) vertex set to be a homogeneous, intensity-$\lambda$ Poisson point process on the unit square $[0,1]^2$ in the plane. Our quantity of interest is the
\emph{total Euclidean edge-length} $L_\lambda$ of the MDSF constructed on the Poisson points, that is, the sum of the Euclidean lengths corresponding to each edge in the graph. 
Our goal is to study distributional asymptotics for the total edge-length, in the $\lambda \to \infty$ limit. For the ordinary NNG (with no constraints) a classical result due to
Bickel and Breiman~\cite{bickel-breiman} and Avram and Bertsimas~\cite{avram-bertsimas} says that $L_\lambda - \Exp L_\lambda$ has a Gaussian limit as $\lambda \to \infty$
(an alternative approach was also proposed by Ramey~\cite{ramey}). Underlying this result is the fact that contributions from individual edge-lengths are asymptotically negligible, and there is a \emph{local dependence} in that edge-lengths associated with distant vertices are almost independent.

Introducing directional constraints leads to new phenomena, 
due to the possibility of long edges near to the boundaries of the domain, when those boundaries and the directional constraints align. One sees this, for example, in Bhatt and Roy's original ``south-west'' MDSF~\cite{bhatt-roy}, in which the permitted cone is a quadrant. Figure~\ref{fig:nng-mdsf} gives a picture. Here, due to the boundary effects, a result of Penrose and Wade~\cite{pw-south-west} shows that $L_\lambda - \Exp L_\lambda$ has a distributional limit as $\lambda \to \infty$ which is \emph{non-Gaussian}, but which can be decomposed as an independent sum of a Gaussian component (associated with contributions in the ``bulk'') and a non-Gaussian component that is characterized by a distributional fixed-point equation associated with a one-dimensional approximation in the neighbourhood of the boundary.
The situation is similar in the case where the permitted cone is a half-plane, which leads to the ``south'' MDSF~\cite{pw-south}.

\begin{figure}[!h]
\centering
\includegraphics[width=0.48\textwidth, height = 0.48\textwidth]{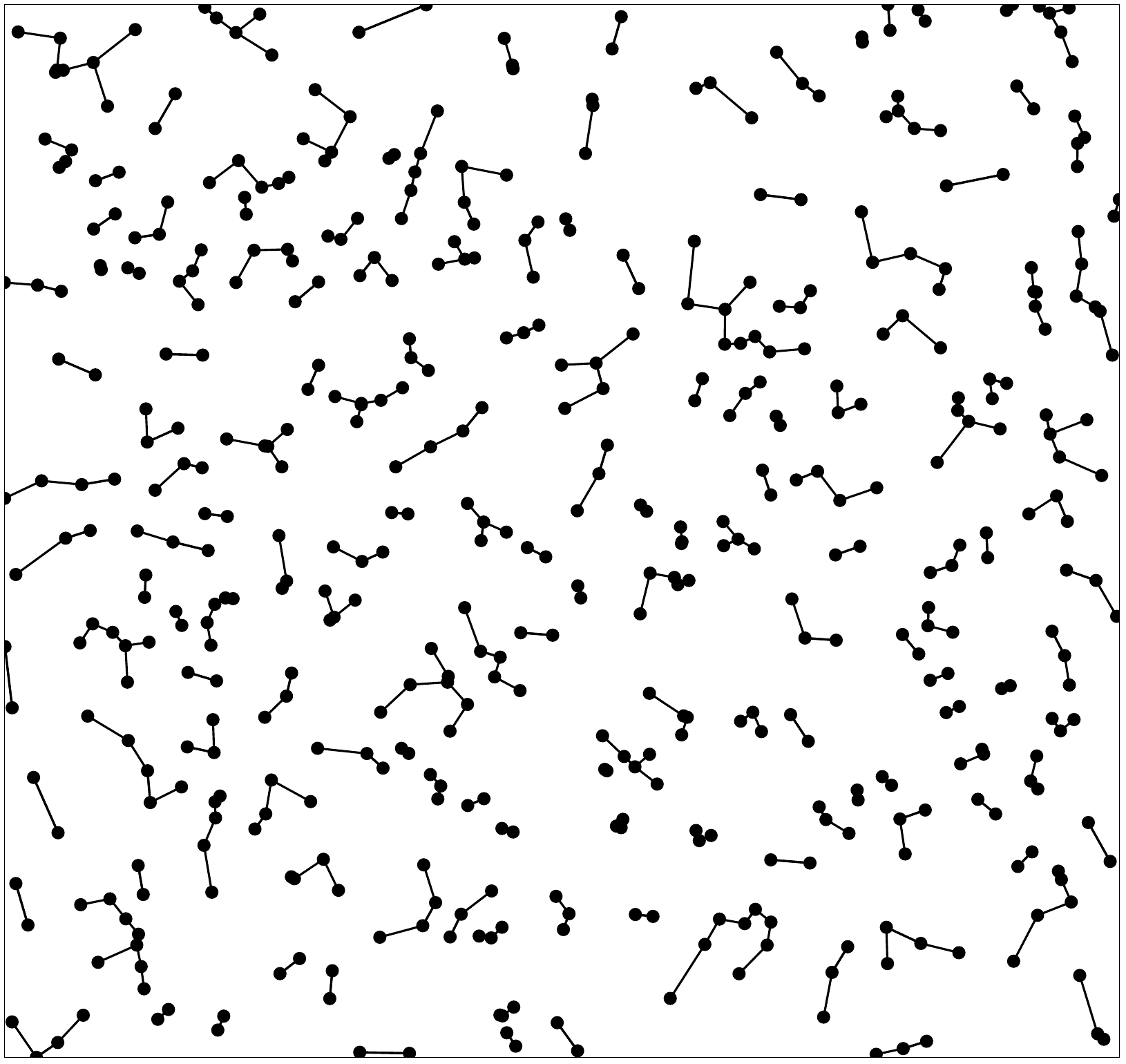}
\hfill 
\includegraphics[width=0.48\textwidth, height = 0.48\textwidth]{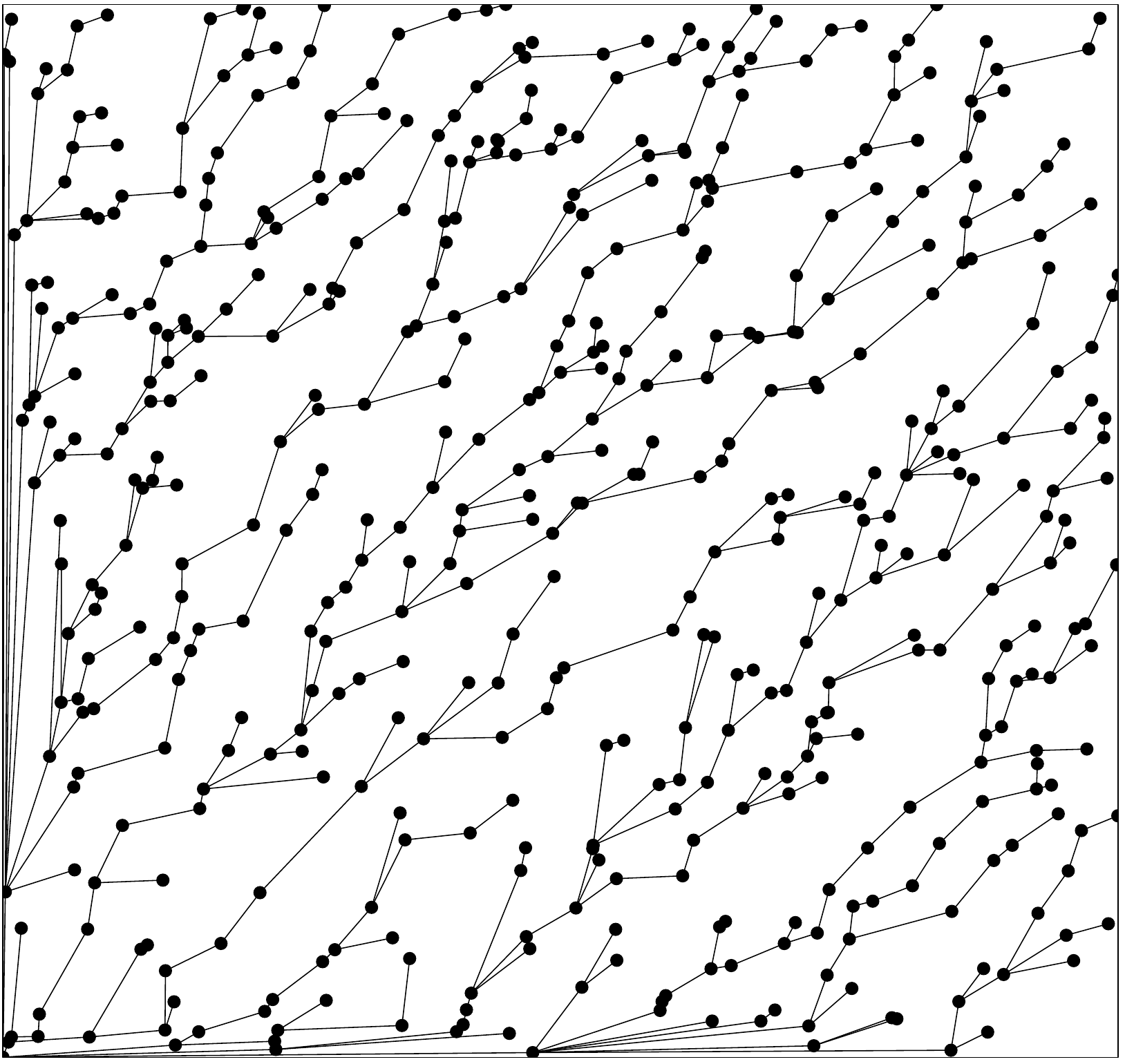}
\caption{Simulated realizations of NNG (\emph{left}) and ``south-west'' MDSF (\emph{right}) in which every point is joined by an edge to its nearest-neighbour with no restriction (\emph{left}) and with the restriction that the neighbour must be below and to the left (\emph{right}). One can see  much longer edges in the MDSF, which accounts for the disruption of Gaussian limit behaviour. Another difference is that degrees of vertices in the MDSF do not, generally, have a uniform upper bound.}
\label{fig:nng-mdsf}
\end{figure}

For general directional constraints in the plane, first-order behaviour was quantified by the law of large numbers of~\cite{wade-llns}. The goal of this paper is to establish second order (i.e., weak convergence) results for
the general case, building a framework to encompass both all the new ``non-rectangular'' cases as well as the phenomena observed in the ``south'' and ``south-west'' cases~\cite{pw-south,pw-south-west}. Our contribution is thus both to settle a conjecture posed in~\cite[\S 7.1.3]{wade-thesis}, as well as to provide a framework that may provide the basis for further extensions: see Section~\ref{sec:discussion} below for some open directions.

\subsection{Minimal directed spanning forest}
\label{sec:mdsf-cones}

The MDSF that we study can be defined on any locally finite subset of $\R^2$. The directional restriction on edges is specified via a cone of permitted directions, parametrized by two angles, $\theta \in [0, 2\pi)$ and $\phi \in (0,  \pi)$. Throughout the paper, our convention is that angles are measured 
anticlockwise from the vertically upward direction.
 
 \begin{definition}[Cone of directions] 
 \label{def:half-cone-closed}
For $\theta \in [0, 2\pi)$ and $\phi \in (0, \pi)$, let $\cCtp(\bx)$ denote the closed cone with apex at $\bx \in \R^2$
consisting of all rays from $\bx$ with angular range given by the  interval $[\theta, \theta+\phi]$. Also define $\cCotp (\bx ) := \cCtp(\bx) \setminus \{ \bx \}$, the cone excluding the apex.
 \end{definition}

\begin{definition}[Directed nearest neighbours]
\label{def:directed-nn}
Fix $\theta \in [0,2\pi)$, $\phi \in (0,\pi)$.
    Let $\cX \subseteq \R^2$ be locally finite and suppose that $\bx \in \R^2$. 
The \emph{directed nearest-neighbour distance} from $\bx$ in $\cX$ (note that $\bx$ need not be an element of $\cX$) is defined to be
\[ 
\cDtp(\bx, \cX) := \begin{cases}
  \inf\left\{\|\by-\bx\|: \by \in \cX \cap \cCotp(\bx)  \right\} & \text{if } \cX \cap \cCo_{\theta, \phi } (\bx )   \neq \emptyset, \\
  0 & \text{if } \cX \cap \cCo_{\theta, \phi } (\bx )   = \emptyset.
\end{cases}
\]
By local finiteness of $\cX$, if $\cX \cap \cCotp (\bx) \neq \emptyset$, then $\cDtp(\bx, \cX) >0$ is
    attained by at least one 
    $\by \in \cX \cap \cCotp(\bx)$,
    and we define a \emph{directed nearest neighbour}
    $N_{\theta, \phi}(\bx, \cX)$ of $\bx$ in $\cX$ to be one such $\by$, chosen arbitrarily if there is more than one. On the other hand,
    if $\cX \cap \cCotp(\bx) = \emptyset$, set 
       $N_{\theta, \phi}(\bx, \cX) := \bx$. In either case, it holds that
    \[
   \| \bx-N_{\theta, \phi}(\bx, \cX)\| = \cDtp(\bx, \cX).
    \]
\end{definition}

\begin{definition}[Minimal directed spanning forest]    
\label{def:mdsf}
Fix $\theta \in [0,2\pi)$, $\phi \in (0,\pi)$.
    Let $\cX \subseteq \R^2$ be locally finite. The MDSF on $\cX$ is the directed graph obtained by joining each point $\bx \in \cX$ by an edge to its directed nearest neighbour $N_{\theta, \phi}(\bx, \cX)$ when the latter is not equal to $\bx$.

    Let $\cL(\cX)$ be the \emph{total edge length}  of the MDSF on $\cX$, i.e.,
\begin{equation}
\label{eq:edge-length}
    \cL(\cX) := \sum_{\bx \in \cX} \cDtp(\bx,\cX) = \sum_{\bx \in \cX}\| \bx- N_{\theta, \phi}(\bx, \cX)\| .
\end{equation}
\end{definition}

Our limit theory for the MDSF will apply to all cases $\theta \in [0,2\pi)$ and $\phi \in (0,\pi)$, with different behaviour among 3 distinct classes of parameters, depending on whether two, one, or none of the extremal directions of the cone align with the directions of the sides of the square, i.e., the set $\frac{\pi}{2} \Z$. The classes of cone geometries are partitioned as follows.

\begin{definition}[Cone classification]
\label{def:cone-types}
For $\theta \in [0,2\pi)$, $\phi \in (0,\pi)$, we apply the following terminology. 
    \begin{enumerate}[label=(\alph*)]
        \item \label{def:cone-types-a} Angles $(\theta,\phi)$ are called \emph{doubly-aligned} if  $\{ \theta , \theta+\phi \} \subseteq  \frac{\pi}{2}\Z$.
        \item \label{def:cone-types-b} Angles $(\theta,\phi)$ are called \emph{singly-aligned} if exactly one of $\theta, \theta+ \phi$ belongs to $\frac{\pi}{2}\Z$.
        \item \label{def:cone-types-c} Angles $(\theta,\phi)$ are called \emph{unaligned} if neither $\theta$ nor $\theta+\phi$ belongs to $\frac{\pi}{2}\Z$.
    \end{enumerate}
\end{definition}

    Cases~\ref{def:cone-types-b} and~\ref{def:cone-types-c} together we call \emph{non-rectangular} cones, and it is these cases that are the main focus of the present paper; see Figure~\ref{fig:mdsf-examples} for some examples. 
    The doubly aligned case was the subject of~\cite{pw-south-west}, while~\cite{pw-south} dealt with the aligned case where $\phi = \pi$.

 \begin{figure}[!h]
\centering
\includegraphics[width=0.32\textwidth]{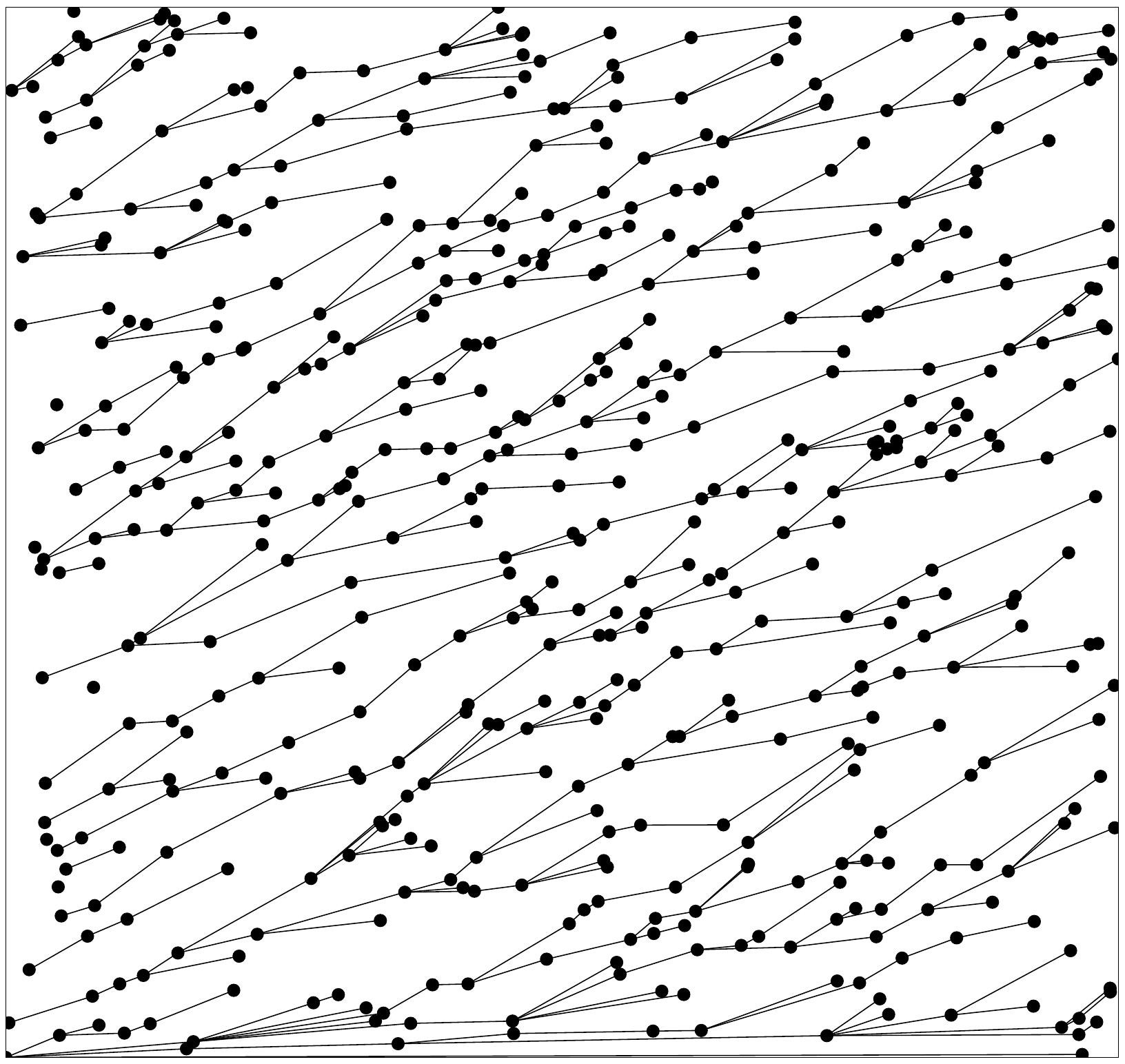}
\,
\includegraphics[width=0.32\textwidth]{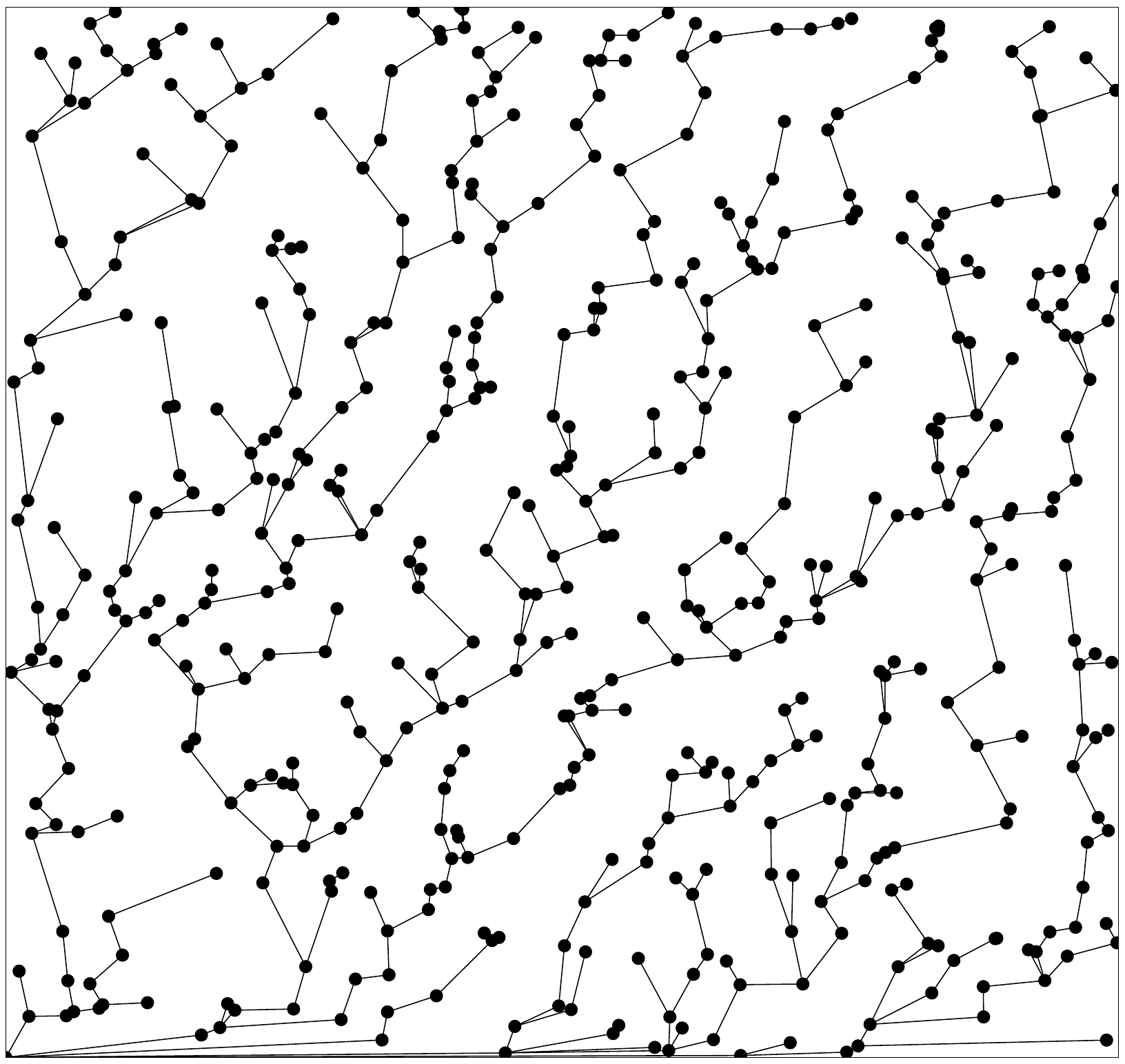}
\,
\includegraphics[width=0.32\textwidth]{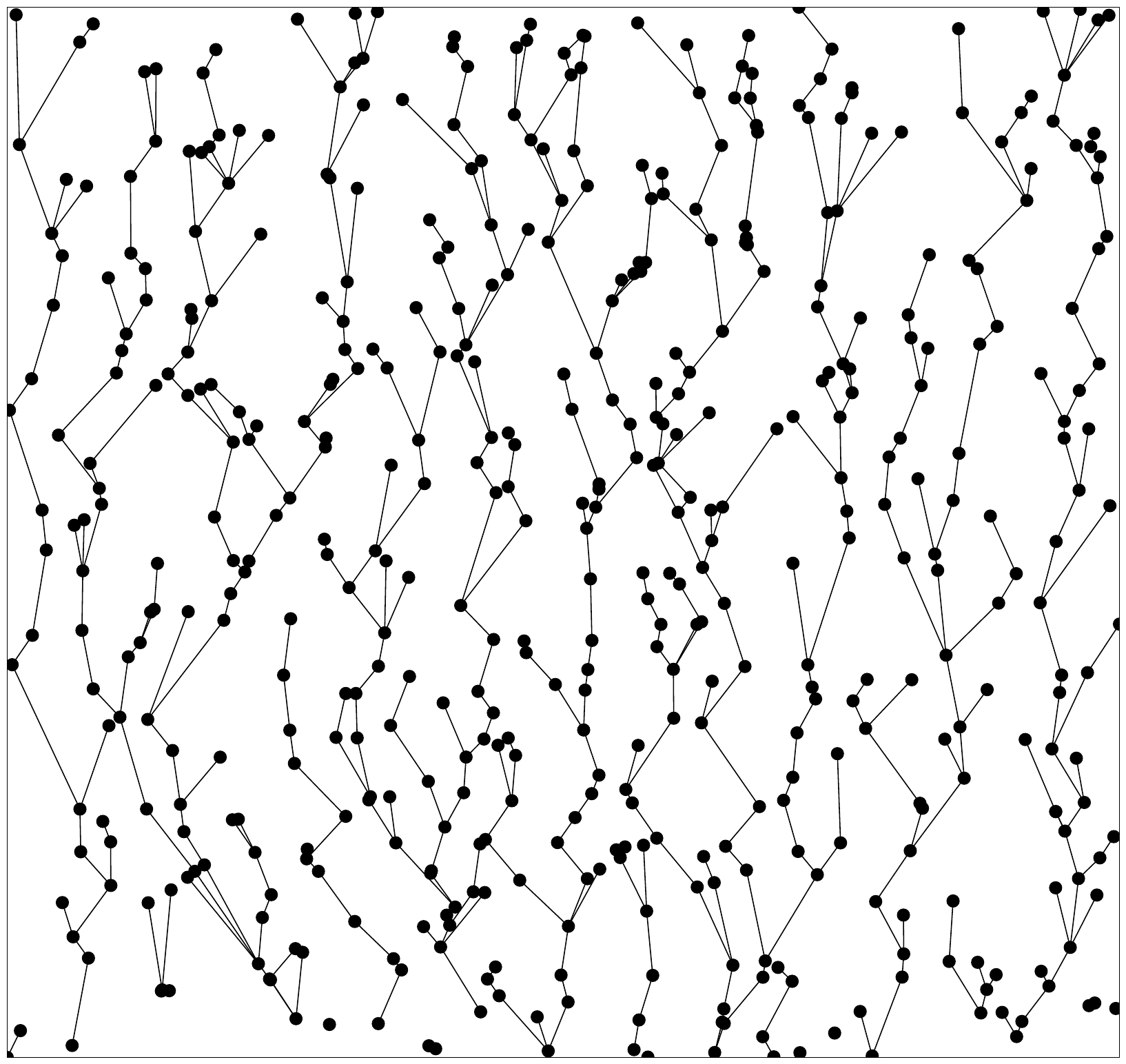}
\caption{Simulated realizations of the MDSF on random points for two singly-aligned cases, the acute cone $(\theta,\phi) = (\frac{\pi}{2},\frac{\pi}{4})$ (\emph{left}) and the obtuse cone $(\theta,\phi) = (\frac{\pi}{2},\frac{3\pi}{4})$ (\emph{middle}), as well as an unaligned case $(\theta,\phi) = (\frac{\pi}{4},\frac{\pi}{2})$ (\emph{right}). In the first two pictures, atypically long edges occur near the bottom edge of the square, as directed neighbours can be much farther away than one would see in the bulk.}
\label{fig:mdsf-examples}
\end{figure}

\begin{remark}[Partial order description]
    \label{rem:partial-order}
Under the given restriction on $\theta, \phi$, there  is a correspondence between the cone description of the MDSF and an equivalent description of the MDSF in terms of a \emph{partial order}. Define relation ``$\preccurlyeq$'' by declaring that 
$\bx \preccurlyeq \by$ if and only if
$\bx \in \cCtp(\by)$. Then $\preccurlyeq$ is  a partial order on $\R^2$ if and only if $0 < \phi < \pi$,
and, for $\phi = \pi$, is also a partial order on subsets of $\R^2$ with distinct projections onto every  direction normal to the boundary; this latter case includes the ``south'' MDSF studied in~\cite{pw-south}. A more general formulation of the MDSF is to permit only  edges that respect a given partial order: see~\cite{bhatt-roy,wade-thesis,pw-survey} for more details. 
\end{remark}

\subsection{Main results}
\label{sec:results}

Let $\cP_\lambda$ denote
a homogeneous Poisson point process on $[0,1]^2$ of intensity $\lambda >0$. 
With $\cL (\cP_\lambda)$ denoting the total edge length of the MDSF, as defined at~\eqref{eq:edge-length}, 
our object of study here is the $\lambda \to \infty$ behaviour of the centred random variable $\tcL_{\lambda} := \cL(\cP_{\lambda})- \Exp \cL(\cP_{\lambda})$. 

Our main results are parts~\ref{thm:generalresult-unaligned} and~\ref{thm:generalresult-singly-aligned} of the following theorem;
part~\ref{thm:generalresult-doubly-aligned}
is due to Penrose and Wade~\cite{pw-south-west} and is included here for completeness. 
We have not included in Theorem~\ref{thm:generalresult} the other previously-known special case, the ``south'' partial order when $\phi = \pi$ and $\theta \in \frac{\pi}{2} \Z$, since in that case it is not the distribution $\cQ$ that appears, but a bi-directional relative~\cite{pw-south}.
We write $\cN (0,1)$ for the standard normal distribution on $\R$,
and use ``$\tod$'' for convergence in distribution.

\begin{theorem}
\label{thm:generalresult}
Denote by
 $\cQ$ the (non-Gaussian) distribution on $\R$ characterized
    as the unique mean-zero, finite-variance solution to~\eqref{eq:fixed-point} below. 
Let $\theta \in [0,2\pi)$, $\phi \in (0,\pi)$.
    There exist  constants $s_{\phi} \in (0,\infty)$, depending only on $\phi$,
    such that if $Q_1$, $Q_2$, $Z$ are mutually independent with $Q_1, Q_2 \sim \cQ$ and $Z \sim \cN(0,1)$, then the following distributional limits hold. 
 \begin{thmenumi}[label=(\roman*)]
        \item 
        \label{thm:generalresult-unaligned} 
        If $(\theta, \phi)$ is unaligned, then, as $\lambda \to \infty$,
        \begin{align}
           \tcL_\lambda\tod s_{\phi}Z.
        \end{align}
          \item 
        \label{thm:generalresult-singly-aligned} 
        If $(\theta, \phi)$ is singly-aligned, then, as $\lambda \to \infty$, 
        \begin{align}
        \label{eq:generalresult-singly-aligned} 
           \tcL_\lambda\tod s_{\phi}Z+Q_{1}.
        \end{align}
         \item
\label{thm:generalresult-doubly-aligned} 
(Penrose and Wade~\cite{pw-south-west}.)
If $(\theta, \phi)$ is doubly-aligned, then, as $\lambda \to \infty$, 
        \begin{align}
           \tcL_\lambda\tod s_{\phi}Z+Q_{1}+Q_{2}.
        \end{align}
    \end{thmenumi}
\end{theorem}

Theorem~\ref{thm:generalresult} shows that
the limit law of the total edge length of the MDSF is determined by the number of cone boundary 
rays aligned with the sides of the square:
unaligned cones give a purely Gaussian limit, while each aligned boundary ray contributes one independent copy of the fixed-point distribution $\cQ$.
In particular, for almost all parameters $(\theta,\phi)$, there is a purely Gaussian limit, but for  singly-aligned or doubly-aligned cones  there are significant boundary effects that are manifest in a non-Gaussian limit. 

As stated in Theorem~\ref{thm:generalresult}, boundary effects 
are given in terms of a distribution $\cQ$ characterized by a fixed-point equation, first appearing, to our knowledge, in~\cite{pw-south-west}.
By ``$\eqd$'' we denote equality in distribution. The characterizing fixed-point equation for $\cQ$ is
 \begin{equation}
    Q \eqd U + UQ^{\prime}+(1-U)Q^{\dprime}+U \log U +(1-U) \log (1-U),  \label{eq:fixed-point} 
\end{equation}
where $Q, Q^{\prime},Q^{\dprime} \sim \cQ$, $U \sim \unif{0}{1}$ and $U, Q', Q''$ are independent. Theorem $3$ of R\"osler~\cite{rosler} implies that there is a unique mean $0$, finite-variance, distribution $\cQ$ corresponding to~\eqref{eq:fixed-point}, and this distribution has finite moments of all orders. The higher order moments of $\cQ$ may be obtained recursively from~\eqref{eq:fixed-point};
for example  $\Var Q  = 2-\frac{\pi^{2}}{6}$ and $\Exp ( Q^3 ) \approx 0.154 >0$~\cite{pw-south-west}, which shows that $\cQ$ is not Gaussian.

\subsection{Discussion and open problems}
\label{sec:discussion}

We make some comments on open problems.

\paragraph{General convex domains.}
Replace the unit square $[0,1]^2$ with a compact, convex planar domain $\cD$. The boundary $\partial \cD$ is partitioned into \emph{faces}, i.e., open line segments of the form $\{ \bu_1 + \lambda (\bu_2 - \bu_1) : \lambda \in (0,1) \}$, between \emph{vertices} $\bu_1, \bu_2$, 
and \emph{extreme points} which are either the endpoint vertices of faces, or comprise curved arcs of the boundary. 

\begin{conjecture}
For general convex domains,  an analogue of Theorem~\ref{thm:generalresult} will hold, with $0, 1$, or $2$ ``boundary terms'' $Q_i$ (scaled by the length of each contributing face) according to whether $\cD$ has exactly $0, 1$, or $2$ faces whose outwards-pointing normal vectors coincide with the inwards-pointing normal vectors to the boundary rays defining the cone $\cCtp$, $\theta \in [0,2\pi)$, $\phi \in (0,\pi)$. \end{conjecture}

\paragraph{Reflex cones.} In this paper, the angular width $\phi$ of the  cone of directions is restricted to $\phi \in (0,\pi)$ (see Definition~\ref{def:half-cone-closed}). This restriction is needed for the cone to specify a plane partial order (see Remark~\ref{rem:partial-order}), and is necessary for Theorem~\ref{thm:generalresult} to hold as stated. However, it is natural to ask about the case  $\phi \in (\pi, 2\pi)$ (cases $\phi=\pi$, if aligned, and $\phi=2\pi$ being already covered by the ``south'' MDSF~\cite{pw-south} and the ordinary NNG~\cite{avram-bertsimas}, respectively). Here we expect new phenomena, particularly if combined with the extension to general convex domains mentioned above.

\paragraph{Higher dimensions.} Consider $[0,1]^d$, $d >2$, and a direction specified by a cone with a fixed central ray and opening angle $\phi \in (0,\pi)$. A natural example is the cone corresponding to the coordinatewise  partial order. Here one expects more complex boundary effects, since, near to an aligned boundary face of the cube, the MDSF is well-approximated by an on-line version in which points arrive sequentially on $[0,1]^{d-1}$. The special case of $\phi = \pi$ was considered in~\cite{pw-south}.

\subsection{Outline of the paper}
\label{sec:outline}

The main aim of the paper is to prove Theorem~\ref{thm:generalresult} parts~\ref{thm:generalresult-unaligned} and~\ref{thm:generalresult-singly-aligned}. 
A key technical element of the proof is to quantify the local dependence of contributions to the total edge length, and the machinery required for this is presented in Section~\ref{sec:partioning-the-square}. In Section~\ref{sec:bulk} we use stabilization methodology and results of Penrose~\cite{mp-gaussian} to show that the contribution to the total edge length of the MDSF from edges in the bulk satisfies a  central limit theorem. In the case of Theorem~\ref{thm:generalresult}\ref{thm:generalresult-unaligned}, where the limit is purely Gaussian, 
this completes the proof. For Theorem~\ref{thm:generalresult}\ref{thm:generalresult-singly-aligned} we need to also control the boundary contributions, which is the subject of Section~\ref{sec:boundary-effects}, where we show that the contribution from edges close to the boundary converges in distribution to a non-normal random variable, whose distribution can be characterized by a fixed-point equation. The proof for the singly-aligned case is completed in Section~\ref{sec:asymptotic-independence}, where we show the asymptotic independence between the two contributions.

\section{Partitioning the square: companions, boundary, and bulk}
\label{sec:partioning-the-square}

\subsection{Introduction}

The goal of this section is to define sub-regions of the unit square for our later analysis: a \emph{boundary}, where there may be long edges, a \emph{bulk}, where there is local dependence and controllable interaction with the boundary, and an \emph{intermediate region} separating the two. It is convenient to build these regions on a grid with a small mesh. To quantify local dependence, we introduce the concept of \emph{companion} tiles of a grid tile, which give locations for candidate nearest neighbours. The geometry of the companions and the boundary influences the choice of the bulk. This section is arranged as follows. 
In Section~\ref{sec:tiles-and-companions}, we introduce the concepts of tiles and companions, and quantify their role in local dependence. In  Section~\ref{sec:identifying-companions}, we   identify some convenient  choices of companions for the different cone geometries.
In Section~\ref{sec:unit-square}, we describe the concept of a compatible bulk and finally, in Definition~\ref{def:3-regions}, complete the construction of the three regions described above.

\subsection{Tiles and companions}
\label{sec:tiles-and-companions}

In this section we introduce a construction for dealing with local dependence,
which involves dividing the whole of $\R^2$ into a grid of boxes, which we call `tiles'. 
We will associate to a `tile' one or more `companions', with the property that a point in a companion is necessarily a candidate nearest-neighbour of a point in the tile, hence bounding the dependency radius.
The  details will depend on the geometry of the cones and corresponding 
  choice of companion (we see this in~Section~\ref{sec:identifying-companions}). First we give the formal definition of tile.

Fix a scale parameter $S$, which will always be the reciprocal of a positive integer. 

\begin{definition}[Tiles]\label{dfin-tiling-in-R^2}
For every $p=(p_1,p_2) \in \Z^{2}$ we define the \emph{tile}
\[ T(p) := [(p_1-1)S,p_1S]\times[(p_2-1)S,p_2S ] \]
and for any $V \subseteq \Z^{2}$ define $T(V) := \cup_{q = (q_{1},q_{2}) \in V} T(q)$. 
\end{definition}

\begin{definition}\label{finiteS_{p,R}}
    Let $\cS_{p,R}$
    be the ball of radius $R$ such that $\cS_{p,R} = \{q \in \Z^{2}: d(p,q) \leq R\}$, where $d(p,q) := \max (|p_{1}-q_{1}|, |p_{2}-q_{2}|)$ denotes the $\ell^\infty$-distance on $\Z^2$.
\end{definition}

The next result says that if two points, one in each of two tiles, are close, then the tiles are close.

\begin{lemma}
\label{lem:proof1.5.1.}
Suppose $\|\bz-\bz^{\prime}\| \leq CS$ for $\bz \in T(p)$, $\bz^{\prime} \in T(q)$ and $C > 0$. Then $d(p,q) \leq \floor*{C}+1$.
\end{lemma}

\begin{proof}
Let $x$ be the horizontal (Euclidean) distance between $\bz$ and $\bz^{\prime}$, and $y$ be the vertical (Euclidean) distance between $\bz$ and $\bz^{\prime}$. Then   $x \leq CS$ and $y \leq CS$.
 Suppose for a contradiction that $|p_1-q_1| \geq \floor*{C}+2$. This implies that there must be a vertical strip of at least $\floor*{C}+1$ squares wide separating $p$ and $q$, which means the horizontal distance $x$ is at least $(\floor*{C}+1)S$,  
 which contradicts our assumption $x\leq CS$.
Similarly, $y \leq CS$ implies $|q_2-p_2|\leq \floor*{C}+1$. Hence, $d(p,q)= \max\{|p_1-q_1|,|p_2-q_2|\} \leq \floor*{C}+1$.
\end{proof}

Now, we define the concept of companion;
recall  $\cCtp(\bx)$ from Definition~\ref{def:half-cone-closed}.

\begin{definition}[Companion]\label{def:anycone}
    For $\theta \in [0,2\pi)$ and $\phi \in (0,\pi)$, for each $p \in \Z^{2}$, call $q \in \Z^{2}$ a \emph{companion} of $p$ if for all $\bx \in T(p)$, it holds that $T(q) \subseteq \cCtp(\bx)$.  
\end{definition}
Note that, $q$ is a companion of $p$ if and only if $q-p$ is a companion of $\0 = (0,0)$.

\begin{definition}[Ball of stabilization]
\label{def:mu,qa}
For $r= (r_1,r_2)$   a companion of $\0$, define 
    $\rho(r) := |r_1|+|r_2|+2$. 
    For $p \in \Z^2$ define 
  \begin{equation}
    \cS_{p,\rho(r)} =\{ q \in \Z^2 :d(q,p) \leq |r_1|+|r_2|+2 \}, \nonumber
\end{equation}
the \emph{ball of stabilization} associated with $p$ and $r$.
Note that, $\cS_{p,\rho(r)}=p+\cS_{\0,\rho(r)}$.
\end{definition}

The next result shows how companions lead to local dependence. Recall by Definition~\ref{def:directed-nn} that $N_{\theta, \phi}(\bx, \cX) $ is a nearest neighbour of $\bx$ on the locally finite set $\cX$.

\begin{proposition}
       \label{prop:geom}
    Suppose $r \in \Z^{2}$ is a companion of $\0$. Then
    for any locally finite set $\cX \subseteq \R^2$, $p \in \Z^{2}$, and $\bx \neq \by \in \cX$ with $\bx \in T(p)$, $\by \in T(p+r)$, we have $N_{\theta,\phi}(\bx,\cX) \in T( \cS_{p,\rho(r)})$. 
\end{proposition}

\begin{proof}
    Take $\0$ with companion $r$. Suppose $\bx,\by \in \cX$ with $\bx \in T(p)$, $\by \in T(p+r)$. By Definition~\ref{def:anycone}, it holds that $T(p+r) \subseteq \cCtp(\bx)$ and hence $\by \in \cCtp(\bx)$, so $\by$ is a candidate nearest-neighbour of $\bx$. Moreover, $\|\by-\bx\| \leq S\sqrt{(|r_1|+1)^{2}+(|r_2|+1)^{2}}$. Also
    \begin{align}
        (|r_1|+|r_2|+3/2)^{2} &= r_1^{2}+3|r_1|+9/8 +r_2^{2}+3|r_2|+9/8+2|r_1||r_2| \nonumber\\ 
        & > (|r_1|+1)^{2}+(|r_2|+1)^{2}, \nonumber
    \end{align}
    implying $\|\by-\bx\| \leq (|r_1|+|r_2|+3/2)S$. %, see Figure~\ref{fig:my-geom12}. 
    Since $\bx$ has at least one candidate nearest-neighbour (namely $\by$), by Definition~\ref{def:directed-nn}, we have 
    \begin{align}
        \cDtp(\bx, \cX)= \| \bx-N_{\theta, \phi}(\bx, \cX)\|  \leq \|\bx - \by\| \leq (|r_1|+|r_2|+3/2)S \nonumber.
    \end{align}
    Finally, let $q \in \Z^{2}$ be such that $N_{\theta, \phi}(\bx, \cX) \in T(q)$. Then Lemma~\ref{lem:proof1.5.1.} implies $d(p,q) \leq |r_1|+|r_2|+2$, 
hence we verify that $q \in  \cS_{p,\rho(r)}$, and so $N_{\theta,\phi}(\bx,\cX) \in T(q) \subseteq T( \cS_{p,\rho(r)})$.
    \end{proof}

The next result, 
which follows from Proposition~\ref{prop:geom}, 
shows that if a tile and its companion each contain a point,
the distance from the point in the tile to its nearest neighbour is not affected by any changes outside its ball of stabilization.

\begin{corollary}
    \label{def:locall-finite-set}
  Let $r \in \Z^{2}$ be a companion of $\0$. Let $\cX \subseteq \R^2$ be  locally finite, and $\bx \neq \by \in \cX$ satisfy $\bx \in T(p)$ and $\by \in T(p+r)$ for some $p \in \Z^{2}$. Then,
    for any locally finite $\cY \subseteq \R^2 \setminus  T(\cS_{p,\rho(r)})$ 
    \begin{equation}
        \cDtp(\bx, \cX) = \cDtp(\bx, \cX \cap T( \cS_{p,\rho(r)}) ) = \cDtp(\bx, ( \cX \cap T( \cS_{p,\rho(r)}) ) \cup \cY ). \nonumber
    \end{equation}
\end{corollary}
\begin{proof}
   Proposition~\ref{prop:geom} implies that $N_{\theta, \phi}(\bx,( \cX \cap T(\cS_{p,\rho(r)})) \cup \cY)$ can not be in $\cY$ since $T(\cS_{p,\rho(r)}) \cap \cY = \emptyset$, so $\cD_{\theta, \phi}(\bx, (\cX \cap T(\cS_{p,\rho(r)}) )\cup \cY) = \cD_{\theta, \phi}(\bx, \cX \cap T(\cS_{p,\rho(r)}))$ for any $\cY \subseteq \R^2 \setminus T(\cS_{p,\rho(r)})$. Moreover, the identity $\cD_{\theta, \phi}(\bx, \cX) = \cD_{\theta, \phi}(\bx, \cX \cap T(\cS_{p,\rho(r)}))$ is just the special case of $\cY= \cX \setminus  T( \cS_{p,\rho(r)})$. The statement follows from the definition of $\cDtp(\bx, \cX)$.
   \end{proof}

\subsection{Identifying companions for different cone geometries}
\label{sec:identifying-companions}

The previous section showed how companions can be used to quantify local dependence. In this section, we identify explicit convenient choices of companions,
depending on the different cone geometries.
There are five cases: 
\begin{enumerate}[label=(\alph*)]
    \item  singly-aligned, obtuse case,
    \item  singly-aligned, acute case,
    \item  unaligned, contains no axes,
    \item unaligned, contains one axis,
    \item  unaligned, contains two axes.
\end{enumerate}
Proposition~\ref{prop:cone-companions}   identifies companions in each case,
which by symmetries of the square reduce to the following. We will from this point onwards, without loss of generality, assume that for singly-aligned cases we have $\theta = \frac{\pi}{2}$ and for unaligned cases we have $0 < \theta < \frac{\pi}{2}$.    

\begin{proposition}\label{prop:cone-companions}
    \begin{thmenumi}[label=(\alph*)]
        \item\label{prop:cone-companions-a} \emph{(Obtuse):} Let $\theta = \pi/2$, $\phi >\pi/2$. Set $h = \cot(\phi- \frac{\pi}{2})$. Then $(0, -(\floor*{h}+2))$ is a companion of $\0$.
        \item\label{prop:cone-companions-b}  \emph{(Acute):} Let  $\theta = \pi/2$, $\phi < \pi/2$. Set $h = 2\cot \phi$. Then $(-(\floor*{h}+2),-1)$ is a companion of~$\0$.
        \item \label{prop:cone-companions-c} \emph{(Unaligned case 0):} Let $0< \theta < \frac{\pi}{2}$ and $\theta+\phi < \frac{\pi}{2}$. Then there exists a companion $r$ of $\0$ with $d(\0,r) \leq \big\lfloor {\frac{3}{\sqrt{2}}(1+\csc(\phi/2))} \big\rfloor +1$.
        \item \label{prop:cone-companions-d} \emph{(Unaligned case 1):} Let $0< \theta < \frac{\pi}{2}$ and $\frac{\pi}{2}< \theta+\phi < \pi$. Set $h = \max(\cot(\frac{\pi}{2}-\theta),\cot(\theta + \phi - \frac{\pi}{2}))$. Then,  $(-(\floor*{h}+2),0)$ is a companion of $\0$.
 \item \label{prop:cone-companions-e} \emph{(Unaligned case 2):} Let $0< \theta < \frac{\pi}{2}$, $\theta+\phi > \pi$, where $\phi < \pi$. Set
  $h_1 = \cot(\frac{\pi}{2} - \theta)$ and $h_2 = \cot(\theta + \phi -\pi)$.
 Then $(-(\floor*{h_{1}}+2),0)$ and $(0,-(\floor*{h_{2}}+2))$ are both companions of $\0$.
    \end{thmenumi}
\end{proposition}
\begin{proof}
For part~\ref{prop:cone-companions-a}, we need to show that $r = (0, -(\floor*{h}+2))$ is a companion of $\0$ for $h = \cot (\phi-\frac{\pi}{2})$. It is enough to show that for every $\bx = (x_1, x_2) \in T(\0)$ and every $\bz = (z_1, z_2) \in T(r)$, $\bz \in  \cCtp(\bx)$. Observe $|x_{1} - z_{1}|  \leq S$ and $|x_{2}- z_{2}| \geq (\floor*{h}+1)S$. Let $\alpha$ be the angle of ray from $\bx$ to $\bz$ measured anticlockwise relative to the downward vertical. Then,
\begin{align}
    \left|\tan \alpha \right| = \frac{|x_{1}-z_{1}|}{|x_{2}-z_{2}|} \leq \frac{1}{\floor*{h}+1} < \frac{1}{h} = \tan\left(\phi - \frac{\pi}{2}\right). \nonumber
\end{align}
%since $\frac{1}{\floor*{h}+1} < \frac{1}{h} = \tan(\phi - \frac{\pi}{2})$. 
In other words, angle $\alpha$ is contained in $(-(\phi - \frac{\pi}{2}),\phi - \frac{\pi}{2})$, which implies that $\bz \in  \cCtp(\bx)$.
 
For part~\ref{prop:cone-companions-b}, we need to show that $r = ( -(\floor*{h}+2),-1)$ is a companion of $\0$. It is enough to show that
for every $\bx = (x_1, x_2) \in T(\0)$ and every $\bz =(z_1,z_2) \in T(r)$, $\bz \in  \cCtp(\bx)$. Observe now $|x_{1} - z_{1}|  \geq (\floor*{h}+1)S$ and $|x_{2}- z_{2}| \leq 2S$. Let $\alpha$ be the angle of ray from $\bx$ to $\bz$ measured anticlockwise relative to the leftward horizontal. Then,
\begin{align}
    \left|\tan \alpha \right| \leq \frac{|x_{2}-z_{2}|}{|x_{1}-z_{1}|} \leq \frac{2}{\floor*{h}+1} < \frac{2}{h} = \tan \phi, \nonumber
\end{align}
by the definition of $h = 2 \cot \phi$. Then $|\alpha| < \phi$, which implies $\bz \in  \cCtp(\bx)$.

For part~\ref{prop:cone-companions-c}, we make use of the fact that if $\by \in \cCtp(\bx)$, then $\bx \in \cC_{\pi +\theta, \phi}(\by)$, since $\cC_{\pi +\theta, \phi}(\by)$ is a rotation by $\pi$ of the cone $\cCtp(\by)$.%, see Figure~\ref{fig:unaligned-case2}.

Consider tile $T(\0)$ and the circumcircle $C_{1}$ of radius $\frac{S}{\sqrt{2}}$ that contains $T(\0)$. Let $\by \in \R^2$ be a point such that $\cC_{\pi +\theta, \phi}(\by)$  contains $C_{1}$, and hence all points in $T(\0)$. We can choose $\by$ so that $\sup_{\bx \in T(\0)} \|\bx-\by\| \leq \ell = \frac{S}{\sqrt{2}}(1+\csc
(\phi/2))$
since $\|\bx- \by\| \leq \|\bx- \bz_{1}\| + \|\bz_{1}- \by\|$ where $\bz_{1}$ is the centre of the $C_{1}$, and $\by$ can be chosen so that $\|\bz_{1}- \by\| = \frac{S}{\sqrt{2}}\csc
(\phi/2)$. By the first statement, this means that $\by \in \cCtp(\bx)$ for all $\bx \in T(\0)$, and hence $\cCtp(\by) \subseteq \cap_{\bx \in T(\0)} \cCtp(\bx)$. Hence ${r}$ is a companion of $\0$ if $T(r) \subseteq \cCtp(\by)$. Let $C_{2}$ be a circle inside $\cCtp(\by)$ large enough to guarantee that it contains some tile $T(r)$. The circle of radius $\sqrt{2}S$ centred at $\bz_{2}$ with distance $\|\bz_{2}- \by\| = \sqrt{2}S\csc(\phi/2)$ will achieve this.
%, see Figure~\ref{fig:unaligned-case2}.
Then for any $\bw \in T(r)$ and $\bx \in T(\0)$, we have 
\begin{align}
    \|\bw-\bx\| & \leq  \|\bw-\bz_{2}\| + \|\bz_{2} - \by\| + \|\by - \bx\|
    \leq \frac{3}{\sqrt{2}}S(1+\csc(\phi/2)).\nonumber
\end{align}
Hence by Lemma~\ref{lem:proof1.5.1.}, we obtain the claimed bound
\begin{equation}
    d(\0,r) \leq \floor*{\frac{3}{\sqrt{2}}(1+\csc(\phi/2))}+1. \nonumber
\end{equation}
  
For part~\ref{prop:cone-companions-d}, we use the fact that $ \cCtp(\bx) = \cC_{0,\theta+ \phi}(\bx) \cap \cC_{\theta, \pi-\theta}(\bx)$ where both $\cC_{0,\theta+ \phi}(\bx)$ and $\cC_{\theta, \pi-\theta}(\bx)$ are singly-aligned cones with obtuse angles at $\bx$. Hence we can deduce using part~\ref{prop:cone-companions-a} that for $\cC_{0,\theta+ \phi}(\bx)$, any square $(-r_{1},0)$ is a companion of $\0$ when $r_{1} \geq \floor*{\cot(\theta+ \phi-\frac{\pi}{2})}+2$, and for $\cC_{\theta, \pi-\theta}(\bx)$, any square $(-r_{1},0)$ is a companion of $\0$ when $r_{1} \geq \floor*{\cot(\frac{\pi}{2}-\theta)}+2$. Therefore $r =(-(\floor*{h} +2),0)$ is a companion of $\0$ for both cones, for $h$ claimed in the proposition, and hence is a companion of $\0$ for their intersection.
 
  For part~\ref{prop:cone-companions-e}, we use the fact that  $ \cCtp(\bx) = \cC_{\frac{\pi}{2},\theta+ \phi-\frac{\pi}{2}}(\bx) \cup \cC_{\theta, \pi-\theta}(\bx)$
   where both $\cC_{\frac{\pi}{2},\theta+ \phi-\frac{\pi}{2}}(\bx)$ and $\cC_{\theta, \pi-\theta}(\bx)$ are singly-aligned cones with obtuse angle at $\bx$. 
Then from part~\ref{prop:cone-companions-a} we know  
    \begin{equation}
        r_{2}=\left(0,-\left(\floor*{\cot\left(\theta+ \phi-\tfrac{\pi}{2} - \tfrac{\pi}{2}\right)}+2\right)\right) = (0,-(\floor*{\cot(\theta+ \phi-\pi)}+2)), \nonumber
    \end{equation}
    is a companion of $\0$ for $\cC_{\frac{\pi}{2},\theta+ \phi-\frac{\pi}{2}}(\bx)$, and therefore a companion of $\0$ for $\cCtp(\bx)$. Similarly, we know $r_{1} = (-(\floor*{\cot(\frac{\pi}{2}-\theta)}+2), 0)$ is a companion of $\0$ for $\cC_{\theta, \pi-\theta}(\bx)$, and therefore a companion of $\0$ for $\cCtp(\bx)$. 
\end{proof}

\subsection{Bulks and boundaries}
\label{sec:unit-square}

In this section we define the notion of a compatible bulk and then finalize our definitions of the three regions of the square. Recall by Definition~\ref{def:mdsf} that $\cL(\cX)$ is the total length of MDSF on $\cX \subseteq [0,1]^2$.
We will define two regions of $[0,1]^2$, $\cRbndry$ (where long edges may appear) and $\cRbulk$ (where a local dependence property holds, as we explain more in Section~\ref{sec:bulk}).
In order to discuss the contribution to $\cL(\cX)$ coming from different regions of $[0,1]^2$, we define the following notation.
\begin{definition}\label{regions}
For $\cR \subseteq \R^2$, define 
\begin{equation}
    \cL(\cX,\cR) := \sum_{\bx \in \cX \cap \cR} \cDtp(\bx,\cX). \nonumber
\end{equation}
\end{definition}
For the unaligned cone, we may take $\cRbndry$ to be empty; 
for the singly-aligned cones $\cRbndry$ will be a narrow strip $[0,1]\times[0,\lambda^{-\sigma}]$ along the bottom boundary of $[0,1]^2$. 
Once $\cRbndry$ is chosen, the bulk $\cRbulk$ is chosen to satisfy two conditions:
\begin{enumerate}
    \item \label{cond1} that the terms  $\cDtp(\bx, \cX)$ contributing to $\cL(\cX,\cRbulk)$ exhibit local  dependence, and
    \item \label{cond2} that any nearest-neighbour $N_{\theta,\phi}(\bx, \cX)$ of a  point $\bx \in \cX \cap \cRbulk$ is not in $\cRbndry$.
\end{enumerate}
For the random sets $\cP_{\lambda}$ that we will consider, condition~\ref{cond1}) will allow us to prove a CLT for  $\cL(\cP_{\lambda},\cRbulk)$ and condition~\ref{cond2}) will ensure an asymptotic independence of $\cL(\cP_{\lambda},\cRbulk)$ and $\cL(\cP_{\lambda},\cRbndry)$  (see Section~\ref{sec:asymptotic-independence-proof} for more details). This motivates the definition of compatible bulk given below in Definition~\ref{def:properties}.
Concerning condition~\ref{cond1}),
for $\bx \in \cX \subseteq [0,1]^2$, local dependence of $\cDtp(\bx, \cX)$ can be achieved in two ways. Either, 
we find a companion $p+r$ of $p$, where $p$ satisfies $\bx \in T(p)$, to apply Proposition~\ref{prop:geom} which requires that $T(p+r) \cap [0,1]^2 \neq \emptyset$; alternatively, $N_{\theta,\phi}(\bx, \cX) \in T(\cS_{p(\bx),R})$ may hold simply because all candidate nearest-neighbours are within a bounded distance of $\bx$, for example because $\bx$ is close to the boundary of $[0,1]^2$, and $\cCtp(\bx) \cap [0,1]^2$ has small diameter. We quantify this second case as follows. 

\begin{definition}\label{def:longestEuclidean}
For $S> 0$ and $p \in \Z^2$, define
    \begin{align}
        \nu(p) := \sup\{
        \|\bx-\by\| : \bx \in T(p), \by \in \cCtp(\bx) \cap [0,1]^2
        \}. \nonumber
    \end{align}
\end{definition}

 \begin{lemma}\label{Lemma443}
     Fix $\theta$, $\phi$. Let $\cX \subseteq [0,1]^2$ be a locally finite set. Suppose $\nu(p) \leq \delta S$ for some $\delta > 0$. Then for all $\bx \in \cX \cap T(p)$, we have $N_{\theta, \phi}(\bx, \cX) \in T( \cS_{p,\floor*{\delta}+1}) \cap [0,1]^2$.
 \end{lemma}
 \begin{proof}
     Suppose $\bx \in \cX \cap T(p)$ and consider $N_{\theta, \phi}(\bx, \cX)$ which by definition is a point in $\cCtp(\bx) \cap \cX \subseteq \cCtp(\bx) \cap [0,1]^2$. Therefore, 
  $\| \bx-N_{\theta, \phi}(\bx, \cX)\| \leq \nu(p) \leq \delta S$.  Let $q$ be such that $N_{\theta, \phi}(\bx, \cX) \in T(q)$. By Lemma~\ref{lem:proof1.5.1.}, $d(p,q) \leq \floor*{\delta}+1$, in other words, $q \in \cS_{p,\floor*{\delta}+1}$.
 \end{proof}

\begin{definition}\label{tile-notation}
    Let $\TT_{S}$ be the set of all tiles in $[0,1]^2$, given by
$\TT_{S} := \{ p \in \Z^2: T(p) \subseteq [0,1]^2
        \}$.
        \end{definition}
Note that, for a particular scale $S$ being the reciprocal of a positive integer, $\TT_{S}$ consists  of all $p =(p_{1}, p_{2}) \in \Z^{2}$ satisfying $1 \leq p_{1}, p_{2} \leq 1/S$. 

We now give the definition of compatible bulk. In Definition~\ref{def:properties} below, $\cRbndry$ is an arbitrary region of $[0,1]^2$ that we will later set to be a boundary region of $[0,1]^2$
where we expect to see long edges.
 The conditions~\ref{stat.1} \&~\ref{stat.2}/\ref{stat.3} ensure the conditions~\ref{cond1}) \&~\ref{cond2}) described above are satisfied for $\cRbulk$ and $\cRbndry$. For now, think of $\lambda$ as an arbitrary parameter that determines the region $\cRbndry$ and $S$ as an arbitrary scale. In fact, we will see in Section~\ref{sec:bulk} that $\lambda$ will be the intensity of our Poisson process and $S$ will be chosen as a function of $\lambda$, with $ S \sim \sqrt{\frac{c \log \lambda}{\lambda}}$.

\begin{definition}[Compatible bulk] \label{def:properties}
     Fix $\theta, \phi, S$. Let $C_{0} \subseteq \Z^{2}$ be a finite (non-empty) set of companions of $\0$. Let $\cRbndry \subseteq [0,1]^2$ and let $\delta \geq 0$. 
     Define the following properties of $p \in \TT_{S}$: 
     \begin{enumerate} [label=(\Alph*)]
         \item  \label{stat.1} $p+r \in \TT_{S} $  and  $T( \cS_{p,\rho(r)}) \cap \cRbndry=\emptyset$ for some $r \in C_{0}$.
     \end{enumerate}
     \begin{enumerate}[label=(${B}_{{\arabic*}}$)]
         \item \label{stat.2} $\nu(p) \leq \delta S$.
         \item \label{stat.3} $( \cup_{\bx \in T(p)} \cCtp(\bx)) \cap \cRbndry = \emptyset
         $.
     \end{enumerate}
     Let $\cB \subseteq \TT_{S}$ be any set of squares each satisfying either~\ref{stat.1} or both~\ref{stat.2} and~\ref{stat.3}. Let 
     $\cRbulk = T(\cB) \subseteq [0,1]^2$. We call $\cRbulk$ a \emph{compatible bulk} (for $\theta$, $\phi$, $C_{0}$, $\cRbndry$, $\delta$ and $S$).
\end{definition}

Now in each of the five cone geometries described in Section~\ref{sec:identifying-companions} we specify $\cRbndry$ and a  compatible bulk,
using the companions described in Proposition~\ref{prop:cone-companions}.
 
\begin{lemma}[{Obtuse}]\label{obtuse'lemma} 
    Set $\theta = \frac{\pi}{2}$ and $\phi \in (\frac{\pi}{2}, \pi)$. 
    Define $h$ as in Proposition~\ref{prop:cone-companions}\ref{prop:cone-companions-a}. 
    Let $C_{0} = \{r\} =\{(0,-(\floor*{h}+2))\}$, $\cRbndry=[0,1]\times[0,\lambda^{-\sigma}]$ for $\sigma \in (\frac{1}{2}, \frac{2}{3})$, $\delta = 0$, $R =\floor*{h}+5 $, and $R+1 \leq \frac{1}{S} < \lambda^{\sigma}$. Then, $\cRbulk=[0,1]\times[RS,1]$ is a compatible bulk.
\end{lemma}
\begin{proof}
    Let $\cB$ be the collection of all squares in $\TT_{S}$ satisfying property~\ref{stat.1}, so that by definition $T(\cB)$ is a compatible bulk. 
 We need to show that $T(\cB) = [0,1] \times [RS,1]$.  First, by definition of $r$ and Definition~\ref{def:mu,qa}, we have 
$\cS_{p,\rho(r)} = \cS_{p,\floor*{h}+4}$.  Also,  since $\lambda^{-\sigma} < S$, the region $\cRbndry$ is contained in the bottom row of tiles of $\TT_{S}$.  Now for $p = (p_{1},p_{2}) \in \TT_{S}$ to satisfy~\ref{stat.1},
    we need $T(\cS_{p,\floor*{h}+4}) \cap \cRbndry = \emptyset$ which holds if $p_{2} - \floor*{h}-4 \geq 2$ since this guarantees that none of $\cS_{p,\lfloor h \rfloor +4}$ is in the bottom row of tiles in $\TT_{S}$.
    %(which necessarily contains $\cRbndry$ since $S> \lambda^{-\sigma}$). 
    For~\ref{stat.1} we also need $p+r \in \TT_{S}$ which requires $1 \leq p_{1} \leq \frac{1}{S}$ and $p_{2} - \floor*{h}-2 \geq 1$. Hence,
$\cB = \{
        p \in \Z^2 : 1 \leq p_{1} \leq \frac{1}{S}, \floor*{h}+6 \leq p_{2} \leq \frac{1}{S}
        \}$, 
   and then, recalling Definition~\ref{dfin-tiling-in-R^2}, we have  $T(\cB) = [0,1] \times [(\floor*{h}+5)S,1]$. %, see Figure~\ref{fig:obtuse-angle}.
\end{proof}

\begin{lemma}[Acute]\label{acute'lemma} 
    Set $\theta = \frac{\pi}{2}$ and $\phi \in (0 , \frac{\pi}{2})$. 
    Define $h$ as in 
    Proposition~\ref{prop:cone-companions}\ref{prop:cone-companions-b}.
    Let $C_{0} = \{r\} =\{(-(\floor*{h}+2),-1)\}$ and $\cRbndry=[0,1]\times[0,\lambda^{-\sigma}]$. Set $\delta = (\floor*{h}+2) \sec \phi$, $R = \max (\floor*{h}+6, \floor*{(\floor*{h}+2) \tan \phi} +2)$, and $R+1 \leq \frac{1}{S} < \lambda^{\sigma}$. Then, $\cRbulk=[0,1]\times[RS,1]$ is a compatible bulk.
\end{lemma}
\begin{proof}
Define $  \cB = \{
        p \in \TT_{S} : p_{2} > R
     \}$. We will show that all $p \in \cB$ either satisfy~\ref{stat.1} or both~\ref{stat.2} and~\ref{stat.3}, which implies that $ T(\cB) = \cRbulk$ is a compatible bulk.

     Suppose $p \in \cB$ has $p_{1} > \floor*{h}+2$, then $p+r \in \TT_{S}$ since $p_{1} + r_{1} > 0$ (and $p_{2} + r_{2} > R-1 > 0$). Also, since $\rho(r) = \floor*{h}+5$ and $p_{2} > R \geq \floor*{h}+6$, we have $T(\cS_{p,\floor*{h}+5}) \cap \cRbndry = \emptyset$ (because $p_{2} - \floor*{h}-5> 1$). Therefore $p$ satisfies~\ref{stat.1}.
     
     Otherwise, suppose $p \in \cB$ has $p_{1} \leq \floor*{h}+2$, then Definition~\ref{def:longestEuclidean} implies that $\nu(p) \leq (\floor*{h}+2)S \sec \phi = \delta S$ by choice of $\delta$, so $p$ satisfies~\ref{stat.2}. 
      To prove that $p$ satisfies~\ref{stat.3}, it is enough to show that for any $\bx \in T(p)$ any $\by = (y_1,y_2) \in \cCtp(\bx) \cap [0,1]^2$ has $y_{2} > \lambda^{-\sigma}$ (this forces $\by \notin \cRbndry$). But this follows, since $p_{2}> R \geq \floor*{\ell} +2$ where $\ell := (\floor*{h}+2) \tan \phi$, i.e., $p_{2} \geq \floor*{\ell} +3$, and $\bx = (x_1,x_2) \in T(p)$ means $x_{2} \geq  (p_{2}-1)S \geq (\floor*{\ell} +2)S$ and $x_{2} - y_{2} \leq \ell S \leq (\floor*{\ell } +1)S$. Then $y_{2} \geq S > \lambda^{-\sigma}$ so $\by \notin \cRbndry$, and therefore $p$ satisfies~\ref{stat.3}. Hence $T(\cB) = [0,1] \times [RS,1] = \cRbulk$ is a compatible bulk.
\end{proof}

\begin{lemma}[Unaligned]\label{unaligned'lemma}
   Set $0<\theta < \frac{\pi}{2}$ and $\phi < \pi$ with $\theta+\phi \notin \{\frac{\pi}{2},\pi\}$. Let $\cRbndry= \emptyset$ (no boundary) and let $\cRbulk= [0,1]^2$. 
   There exists a finite set $C_{0}$ of companions of  $\0$, large enough $\delta$ and small enough $S$ such that $\cRbulk$ is a compatible bulk. 
\end{lemma}
\begin{proof}
Since $\cRbndry= \emptyset$, condition~\ref{stat.1} reduces to $p+r \in \TT_{S}$ for some $r \in C_{0}$ and~\ref{stat.3} is trivially satisfied. Hence, we just need to show for all $p \in \TT_{S}$ that either $p+r \in \TT_{S}$ for some $r \in C_{0}$, or $\nu(p) \leq \delta S$. The argument for the $3$ types of unaligned cone are very similar but rely on different parts of Proposition~\ref{prop:cone-companions}, so we prove the $3$ cases sequentially.  

First consider Case 0 (corresponding to Proposition~\ref{prop:cone-companions}\ref{prop:cone-companions-c}).
Set $D = \lfloor {\frac{3}{\sqrt{2}}(1+\csc(\phi/2))} \rfloor +1$, and let $r$ be a companion with $ d(\0,r) \leq D$, which exists by Proposition~\ref{prop:cone-companions}\ref{prop:cone-companions-c}. Let $C_{0} = \{r\}$, let $\frac{1}{S} \geq D$ and define $\cC_{\text{left}} = \{p=(p_1,p_2) \in \Z^{2} : 1 \leq p_1 \leq D, 1\leq p_2 \leq \frac{1}{S}\}$ and $\cC_{\text{top}} = \{p=(p_1,p_2) \in \Z^{2} : 1 \leq p_1 \leq \frac{1}{S}, \frac{1}{S} - (D-1) \leq p_2 \leq \frac{1}{S}\}$.
If $p \in  \TT_{S} \setminus (\cC_{\text{left}} \cup \cC_{\text{top}})$, then $p+r \in \TT_{S}$ because $p+r \in \cS_{p,D}$ and (by the cone geometry) $p+r$ is  above and to the left of $p$.

If $p \in \cC_{\text{left}}$, then $\nu(p) \leq DS \csc \theta$ and if $p \in \cC_{\text{top}}$, then $\nu(p) \leq DS \sec (\theta + \phi)$. Hence taking $\delta = \max(D \csc \theta, D \sec (\theta+\phi))$, we have $\nu(p) \leq \delta S$ for all $p \in \cC_{\text{left}} \cup \cC_{\text{top}}$.

 Next consider Case 1 (corresponding to Proposition~\ref{prop:cone-companions}\ref{prop:cone-companions-d}).
    Set $D = \floor*{h} + 2$, where $h$ is as defined in Proposition~\ref{prop:cone-companions}\ref{prop:cone-companions-d}. Then $r = (-D,0)$ is a companion of  $\0$. Set $C_{0} =  \{r\}$, let $\frac{1}{S} \geq D$ and define 
$\cC_{\text{left}}$
as before. If $p \in  \TT_{S} \setminus \cC_{\text{left}}$, then  $p+r \in \TT_{S}$. 

If $p \in \cC_{\text{left}}$ then $\nu(p) = \max (\nu_{1}(p),\nu_{2}(p))$ where $\nu_{1}(p) = \sup\{
        \|\bx-\by\| : \bx \in T(p), \by \in \cCtp(\bx) \cap [0,1]^2, y_{2} \geq x_{2}
        \}$
and 
$\nu_{2}(p) = \sup\{
        \|\bx-\by\| : \bx \in T(p), \by \in \cCtp(\bx) \cap [0,1]^2, y_{2} \leq x_{2}
        \}$
    and $\nu_{i}(p) \leq DS \sec \phi_{i}$ for each $i = 1,2$, where $\phi_{1}, \phi_{2}$ are angles satisfying $0 < \phi_{i} < \frac{\pi}{2}$ and $\phi = \phi_{1}+ \phi_{2}$. Hence taking $\delta = \max(D \sec \phi_{1}, D \sec \phi_{2})$, we have $\nu(p) \leq \delta S$ for all $p \in \cC_{\text{left}}$.

Finally, consider Case 2 (corresponding to Proposition~\ref{prop:cone-companions}\ref{prop:cone-companions-e}).
Set $D = \max(\floor*{h_{1}} + 2, \floor*{h_{2}} + 2)$, where $h_{1}$ and $h_{2}$ are as defined in Proposition~\ref{prop:cone-companions}\ref{prop:cone-companions-e}. The squares $r_{1} = (-(\floor*{h_{1}} + 2),0)$ and $r_{2} = (0,-(\floor*{h_{2}} + 2))$ are both companions of  $\0$. Set $C_{0} = \{r_{1}, r_{2}\}$, let $\frac{1}{S} \geq D$ and define $\cC_{\text{left}}$ as before. Let $\cC_{\text{bottom}} = \{p=(p_1,p_2) \in \Z^{2} : 1 \leq p_1 \leq \frac{1}{S}, 1\leq p_2 \leq D\}$. If $p \in  \TT_{S} \setminus \cC_{\text{left}}$, then  $p+r_{1} \in \TT_{S}$ and if $p \in  \TT_{S} \setminus \cC_{\text{bottom}}$, then $p+r_{2} \in \TT_{S}$. Hence for all $p \in  \TT_{S} \setminus (\cC_{\text{left}} \cap \cC_{\text{bottom}})$, there is an $r \in C_{0}$ with $p+r \in \TT_{S}$.

If $p \in \cC_{\text{left}} \cap \cC_{\text{bottom}}$, then $\nu(p) = \max (\nu_{1}(p),\nu_{2}(p), \nu_{3}(p))$ where $\nu_{1}(p) = \sup\{
        \|\bx-\by\| : \bx \in T(p), \by \in \cCtp(\bx) \cap [0,1]^2, y_{2} \geq x_{2}
        \}$,
     $\nu_{2}(p) = \sup\{
        \|\bx-\by\| : \bx \in T(p), \by \in \cCtp(\bx) \cap [0,1]^2, y_{1} \geq x_{1}
        \}$, and 
        $\nu_{3}(p) = \sup\{
        \|\bx-\by\| : \bx \in T(p), \by \in \cCtp(\bx) \cap [0,1]^2, y_{1} \leq x_{1}, y_{2} \leq x_{2}
        \}$, and $\nu_{i}(p) \leq DS \sec \phi_{i}$ for each $i =1,2$, where $\phi_{1}, \phi_{2}$ are angles satisfying $0 < \phi_{i} < \frac{\pi}{2}$ and $\phi = \phi_{1}+ \phi_{2} + \frac{\pi}{2}$. Also $\nu_{3}(p) \leq \sqrt{2} DS$. Hence taking $\delta = \max (D \sec \phi_{1}, D \sec \phi_{2}, \sqrt{2}D)$, then $\nu(p) \leq \delta S$ for all $p \in \cC_{\text{left}} \cap \cC_{\text{bottom}}$.      
\end{proof}

The above constructions give  a compatible bulk in every cone geometry. We are thus in a position to  conclude this section by formalizing the definitions of the three regions of the unit square that we use in our subsequent analysis.

\begin{definition}\label{def:3-regions}
Let $(\theta,\phi)$ be singly-aligned. Let $S > \lambda^{-\sigma}$, and $\sigma \in (\frac{1}{2},\frac{2}{3})$, and define
\begin{align}
\cRbulk &:= \left[0,1 \right] \times [RS,1] ;
\label{eq:R-1-def}  \\
\cRint &:= \left[0,1 \right] \times \left[\lambda^{-\sigma}, RS\right]  ;
\label{eq:R-2-def} \\
\cRbndry &:= \left[0,1 \right]  \times \left[0, \lambda^{-\sigma}\right],
\label{eq:R-3-def}
\end{align}
where $R= R(\theta,\phi) \in \ZP$ is as specified in the construction of
the compatible bulk described in  Lemmas~\ref{obtuse'lemma} and~\ref{acute'lemma}, in the obtuse and acute cases, respectively. Otherwise, if $(\theta,\phi)$ is unaligned, we take $\cRbulk = [0,1]^2$ and $\cRint = \cRbndry = \emptyset$, so that $\cRbulk$ is a compatible bulk according to Lemma~\ref{unaligned'lemma}.
In every case, we assume  that $\lambda$ is sufficiently large and $S$ chosen so that the conditions of the relevant Lemma~\ref{obtuse'lemma}--\ref{unaligned'lemma} hold.
\end{definition}

Recall  that $\cP_\lambda$ is a homogeneous Poisson point process on $[0,1]^2$ with intensity $\lambda$. The next definition is to define the total length for all regions (recall Definition~\ref{regions}) in $[0,1]^2$. 

\begin{definition}\label{allregions}
    For $* \in \{ \mathrm{blk}, \mathrm{itm}, \mathrm{bdy} \}$, define 
    $\cL^{*}_{\lambda} := \cL(\cP_{\lambda},\cR^{*}_{\lambda})$,
    and $\widetilde{\cL}_\lambda^{*} := \cL^{*}_{\lambda} - \Exp[\cL^{*}_{\lambda}]$, its centred version.
\end{definition}

\section{Gaussian limits from the bulk}
\label{sec:bulk}

\subsection{Introduction}

In this section we prove a CLT for the contribution to the total edge length in the MDSF from edges coming from the \emph{bulk} (Theorem~\ref{thm:bulk-clt} below). The bulk takes up, asymptotically, all but a vanishing fraction of the area of the unit square. The precise definition of the bulk depends on the parameters $\theta, \phi$ specifying the directional constraint; recall from  Definition~\ref{def:3-regions} that $\cRbulk = \left[0,1 \right] \times [RS,1]$, for the constant $R \in \ZP$ given in Lemmas~\ref{obtuse'lemma},~\ref{acute'lemma}, and~\ref{unaligned'lemma}, and a scale parameter $S$ that we will choose appropriately in terms of $\lambda$. The associated total edge length (see Definition~\ref{allregions}) is $\cLbulk$. Without loss of generality, from this point onwards in the paper we will always assume that if $(\theta, \phi)$ is singly-aligned, then $\theta = \pi/2$.

We now specify the scale parameter $S$ in the definition of the bulk. 
For $c>0$ define
\begin{equation}
    \label{eq:a-b-def}
b(c,\lambda) := \sqrt{\frac{c \log \lambda}{\lambda}}, \text{ for $\lambda \geq 1$, and }
a(c,\lambda) := \min \left[ \frac{1}{2}, b(c,\max(3, \lambda)) \right] \text{ for } \lambda >0. 
\end{equation}
Then set
\begin{equation}
    \label{eq:S-choice}
    S := S(c,\lambda ) = \frac{1}{\lfloor 1/a (c,\lambda) \rfloor}.
\end{equation}
The definitions through~\eqref{eq:a-b-def}--\eqref{eq:S-choice} ensure that $S \sim \sqrt{ c \lambda^{-1} \log \lambda}$ as $\lambda \to \infty$, that $1/S$ is an integer, and that $S$ is well-behaved even when $\lambda$ is small. Note that, for all $\lambda$ sufficiently large, 
$S$ satisfies the conditions of Lemmas~\ref{obtuse'lemma}--\ref{unaligned'lemma}.
 
\begin{theorem}
\label{thm:bulk-clt}
For all $\theta \in [0,2\pi)$, $\phi \in (0,\pi)$ either singly-aligned or unaligned, there exist positive constants $s_\phi^2$ (depending only on $\phi$) and $c_{0}$ such that the following holds.
Take $c > c_0$ and define $S = S (c,\lambda)$ by~\eqref{eq:S-choice}. 
Let $\cRbulk$ be the compatible bulk for $\theta,\phi$ specified in Lemmas~\ref{obtuse'lemma}--\ref{unaligned'lemma}, i.e., $\cRbulk = \left[0,1 \right] \times [RS,1]$, where $R = 0$ (unaligned) or $R = R(\theta,\phi) \in \N$ (singly-aligned).  Then
 \begin{thmenumi}[label=(\roman*)]
    \item As $\lambda \to \infty$,
\begin{equation}
\label{eq:bulk-clt}
\tcLbulk \tod \cN(0,s^2_\phi), \text{ and}
\end{equation}
    \item
\begin{equation}
\label{eq:bulk-variance-limit}
   \lim_{\lambda \to \infty}\Var \cLbulk  = s_{\phi}^{2}. 
\end{equation}
\end{thmenumi}
\end{theorem}

Theorem~\ref{thm:bulk-clt} suffices to prove our main result,
Theorem~\ref{thm:generalresult} in the case of unaligned cones, when the limit is purely Gaussian, and forms a key ingredient in the singly-aligned case, when non-Gaussian boundary contributions must also be considered. 

\begin{proof}[Proof of Theorem~\ref{thm:generalresult}\ref{thm:generalresult-unaligned}]
In the unaligned case, $\cRbulk = [0,1]^2$ (see Lemma~\ref{unaligned'lemma})
and so Theorem~\ref{thm:bulk-clt} implies the result.
\end{proof}

The aim of the rest of this section is to prove Theorem~\ref{thm:bulk-clt}. The proof is presented in 
Section~\ref{sec:clt-proof} below. The proof uses the methodology of \emph{stabilization} and appeals to results of Penrose~\cite{mp-gaussian}; the necessary definitions and technical estimates are presented in    Section~\ref{sec:stabilization}. Our approach is somewhat different from that in~\cite{pw-south-west}, which adapted results from~\cite{penrose-yukich-clt} (the results of~\cite{mp-gaussian} being not yet available).

\subsection{Stabilization estimates}
\label{sec:stabilization}

Stabilization is a well-established framework for quantifying and exploiting local dependence
in problems in geometric probability, with a variety of related formulations that have been successfully applied to many models; see e.g.~\cite{penrose-percolation,MR2165584,penrose-yukich-clt,mp-gaussian,MR2201885,bhatt-roy}, among others. 
We will use the formulation and results from~\cite{mp-gaussian}, and so closely align our notation with that paper. We recall the definitions and results that we need in this section.

We will consider a family $\Omega_\lambda \subseteq [0,1]^2$ of sets that converge to $\Omega := [0,1]^2$, which will determine which Poisson points contribute scores; we will ultimately take $\Omega_\lambda = \cRbulk$, the bulk specified in Theorem~\ref{thm:bulk-clt}, but for now we state only the key properties of $\Omega_\lambda$ that will be used. 
Take 
\begin{equation}
    \label{eq:Omega-lambda} 
    \Omega_\lambda := [0,1] \times [\alpha_{\lambda},1],
\end{equation}
where there are two cases, namely  
 \begin{itemize}
     \item if ($\theta, \phi$) is unaligned, when $\alpha_\lambda \equiv 0$ for all $\lambda$, 
     \item if ($\theta, \phi$) is singly aligned, when $\alpha_{\lambda} \in (0,1/2)$ is any sequence such %\red{that $\alpha_{\lambda} \sqrt{\lambda} \geq 2$ for all $\lambda$}, and
          that $\lim_{\lambda \to \infty} \alpha_{\lambda} = 0$ and $\alpha_{\lambda} > \sqrt{\frac{c \log \lambda}{\lambda}}$ for large enough $\lambda$, where $c> c_{0}$ will be chosen sufficiently large later in this section.
     \end{itemize}
 The latter case includes the choice $\alpha_{\lambda} = RS$ 
 with $S = S(c,\lambda)$ as in Theorem~\ref{thm:bulk-clt}. Indeed,  for $S = 1/ \floor*{ 1 / {a_{\lambda}}}$, we have $Ra_{\lambda} \leq \alpha_{\lambda} = RS = 
 \frac{R}{\floor*{ 1 / {a_{\lambda}}}} \leq
\frac{R}{ (1 / {a_{\lambda}} )-1} = \frac{Ra_{\lambda}}{1-a_{\lambda}}
 \leq 2Ra_{\lambda}$, meaning  $\alpha_{\lambda}$ and $a_{\lambda}$ exhibit, up to multiplicative constants,  the same asymptotics. 

 Recall from Definition~\ref{def:half-cone-closed} that $\cCtp(\bx)$ is a cone with apex at $\bx \in \R^2$ formed as a union of rays from $\bx$ with angle $\alpha \in [\theta,\theta + \phi]$, measured anticlockwise from the vertical direction, and $\cCotp (\bx)$ is the same cone but with the apex at $\bx$ deleted. 
Recall  from Definition~\ref{def:directed-nn} that, 
for $\cX \subset \R^2$ locally finite, and $\bx \in \R^2$,
$\cDtp(\bx, \cX)$ denotes the distance from point $\bx$ to the nearest neighbour to $\bx$ among $\cX \cap \cCotp(\bx)$, unless there is no such neighbour, in which case $\cDtp (\bx, \cX) = 0$. 
The goal of this section is to develop some \emph{stabilization} estimates for $\cDtp$ which quantify its \emph{local dependence} when acting on Poisson points. 

From the underlying norm,  $\cDtp$ inherits  \emph{homogeneity} under scaling of its arguments, i.e.,
\begin{equation}
\label{eq:homogeneous}
    \cDtp(s\bx,s\cX) = s\cDtp(\bx,\cX),  \text{ for all } s > 0. 
\end{equation}
Moreover, 
 $\cDtp$ is \emph{translation invariant}, meaning that
\begin{equation}
\label{eq:translation-invariant}
     \cDtp (\bx, \cX) =  \cDtp (\by + \bx, \by + \cX),   \text{ for all } \by \in \R^2. 
\end{equation}
(Here and elsewhere, for scalar $s>0$, vector $\by \in \R^2$ and $\cX \subseteq \R^2$, $\by + s \cX := \{ \by + s \bx : \bx \in \cX\}$.)

The Poisson process $\cP_\lambda$ has intensity $\lambda$ (which we think of as large, since we are interested in $\lambda \to \infty$), so that the scaled-up process $\lambda^{1/2} \cP_\lambda$ has unit intensity. Thus to study the collection $\cDtp (\bx, \cP_\lambda)$ over $\bx \in \cP_\lambda$ (at least in the bulk $\Omega_\lambda$) on a correctly-calibrated scale, it is useful to define a re-scaled version of $\cDtp$ as follows.
\begin{definition}
\label{def:nn-dist-scaled}
For $\lambda >0$, define $\cDtpl$ for locally finite $\cX \subseteq \R^2$ and $\bx \in \R^2$, 
 \begin{align}
     \cDtpl (\bx,\cX) := \cDtp(\lambda^{1/2}\bx, \lambda^{1/2} \cX) \ind_{\Omega_\lambda}(\bx) = \lambda^{1/2} \cDtp(\bx,\cX)\ind_{\Omega_\lambda}(\bx) , \label{eq:lambda-scaling} 
 \end{align}
where we have used~\eqref{eq:homogeneous} for the last equality in~\eqref{eq:lambda-scaling}.
\end{definition}

With Definition~\ref{def:nn-dist-scaled}, we can then use~\eqref{eq:lambda-scaling} to express the total bulk length as
\begin{equation}
\label{eq:bulk-scaled}
     \cLbulk = \sum_{\bx \in \cP_{\lambda} \cap \Omega_\lambda} \cDtp (\bx, \cP_{\lambda}) = \lambda^{-1/2} \sum_{\bx \in \cP_{\lambda}} \cDtpl(\bx,\cP_{\lambda}),
 \end{equation}
where now, roughly-speaking, the $O(\lambda)$-many terms in the right-most sum in~\eqref{eq:bulk-scaled}
are random-variables of size $O(1)$,
and, for large $\lambda$, each $\cDtpl(\bx,\cP_{\lambda})$ looks roughly like
$\cDtp$ evaluated at a point of an \emph{infinite} unit-intensity Poisson process on $\R^2$.
The key ingredient in establishing central limit theorem and variance asymptotics  is to appropriately quantify the local dependence in the sum on the right-hand side of~\eqref{eq:bulk-scaled}. This is the role of the radius of stabilization that we define next.

Recall that $\Vert \cdot \Vert$ is the Euclidean norm on $\R^2$. Let 
\begin{equation}
    \label{eq:ball-def}
B_r(\bx) := \{ \by \in \R^2 :  \Vert \by-\bx \Vert \leq r \}, 
\end{equation}
 the closed Euclidean ball centred at $\bx \in \R^2$ with radius $r >0$.   The next definition follows~\cite{mp-gaussian}. 

\begin{definition}
\label{def:radius-of-stabilization}
\begin{enumerate}[label=(\roman*)]
    \item\label{def:radius-of-stabilization-i}
The radius of stabilization $R(\bx, \cX, A)$ of $\cDtp$ at $\bx \in \R^2$ with respect to the locally finite set $\cX \subseteq \R^2$ and $A \subseteq \R^2$ is defined to be the smallest $r \in \ZP$ such that
\begin{align}
\label{eq:radius-stabilization}
    \cDtp (\bx, \cX \cap B_r(\bx)) = \cDtp(\bx,(\cX \cap B_r(\bx)) \cup \cY), \text{ for all finite } \cY \subseteq A \setminus B_r(\bx);
\end{align}
if no such $r$ exists, set $R(\bx, \cX, A) = \infty$.
Also set $R (\bx , \cX ) := R (\bx , \cX , \R^2 )$.
\item\label{def:radius-of-stabilization-ii}
For $\lambda >0$ and $\bx \in [0,1]^2$, with $R(\bx,  \cX, A)$  as  in part~\ref{def:radius-of-stabilization-i}, define
\begin{equation}
\label{eq:radius-stabilization-scaled}
    R_\lambda(\bx) := R(\lambda^{1/2} \bx, \lambda^{1/2} \cP_{\lambda}, [0,\lambda^{1/2}]^2) .
\end{equation}
\end{enumerate}
\end{definition}

 It is useful to observe a deterministic upper bound on $\cDtp (\bx, \cX)$:  for $\bx \in [0,1]^2$  define 
    \begin{equation}
    \label{eq:nu-def}
        \nu_{\theta,\phi}(\bx) := \sup\{
        \|\bx - \by\| : \by \in \cCtp(\bx) \cap [0,1]^2
        \}, 
    \end{equation} then $\cDtp (\bx, \cX)  \leq \nu_{\theta,\phi}(\bx)$ for all $\bx \in [0,1]^2 $  and $ \cX \subseteq [0,1]^2$.
Definition~\ref{def:radius-of-stabilization} immediately yields the following properties for $R(\bx, \cX)$ and $R_{\lambda}(\bx)$.

\begin{lemma}
\label{lem:R-to-D-relations}
For all locally finite $\cX \subseteq \R^2$ and all $\bx\in \R^2$, it holds that
\[
   R(\bx, \cX) = 
\begin{cases}
\lceil \cDtp(\bx,\cX )\rceil & \text{if}\ \cDtp(\bx,\cX ) > 0, \\  
 \infty & \text{if}\ \cDtp(\bx,\cX ) = 0. 
\end{cases}
\]
For all $\bx \in [0,1]^2$ and all $\lambda >0$, it holds that
\[
       R_{\lambda}(\bx)
 \begin{cases}
  = \lceil \lambda^{\frac{1}{2}} \cDtp(\bx,\cP_{\lambda})\rceil & \text{if}\  \cDtp(\bx,\cP_{\lambda}) > 0, \\ 
 \leq \lceil \lambda^{\frac{1}{2}} \nu_{\theta,\phi}(\bx) \rceil & \text{if}\ \cDtp(\bx,\cP_{\lambda}) = 0. 
\end{cases}
\]
\end{lemma}

The following monotonicity property for $R$ will be useful.
\begin{lemma}
\label{lem:R-monotonicity}
    For all $A \subseteq \R^2$, all  $\cX \subseteq \R^2$, and all $\bx,\by\in \R^2$, 
   $ R(\bx, \cX, A) \geq R(\bx, \cX \cup \{\by\},A)$. 
\end{lemma}
\begin{proof}
 Fix $\bx,\by\in \R^2$ and $\cX \subseteq \R^2$.
 We show that for each $r \in \ZP$ if condition~\eqref{eq:radius-stabilization} (from Definition~\ref{def:radius-of-stabilization}\ref{def:radius-of-stabilization-i}) holds for the set $\cX$, then it also holds for the set  $\cX \cup \{\by\}$. First suppose~\eqref{eq:radius-stabilization} holds for $\cX$, with $\cDtp(\bx,\cX \cap  B_r(\bx)) =0$. This requires that $((\cX \cap  B_r(\bx) ) \cup \cY) \cap \cCotp(\bx) = \emptyset$ for all finite $ \cY \subseteq A \setminus B_r(\bx)$. Hence, for all finite $\cY \subseteq A \setminus B_r(\bx)$,  
\begin{align*}
    \cDtp(\bx, ((\cX \cup \{\by\}) \cap  B_r(\bx)) \cup \cY) = \cDtp(\bx,   \{\by\} \cap B_r(\bx))
    &= \cDtp(\bx, (\cX\cup \{\by\}) \cap B_r(\bx)).
\end{align*}
 On the other hand, suppose that  $\cDtp(\bx, \cX \cap B_r(\bx)) > 0$, so  $\cX \cap B_r(\bx) \cap \cCotp(x) \neq \emptyset$, hence also
 $(\cX \cup \{\by\}) \cap B_r(\bx) \cap \cCotp (x) \neq \emptyset$. Then for all $\cY \subseteq A \setminus B_r(\bx)$, the nearest neighbour of $\bx$  in $((\cX \cup \{\by\}) \cap B_r(\bx))\cup \cY$ cannot be in $\cY$, so 
\[
    \cDtp(\bx, ( (\cX \cup \{\by\}) \cap B_r(\bx) )\cup \cY ) = \cDtp(\bx,(\cX \cup \{\by\}) \cap B_r(\bx)). 
\]
So, in either case,~\eqref{eq:radius-stabilization} holds  for $\cX \cup \{\by\}$. Hence the infimum over $r$ satisfying~\eqref{eq:radius-stabilization} for $\cX \cup \{\by\}$ is less than or equal to the  infimum of $r$ satisfying condition~\eqref{eq:radius-stabilization} for $\cX$. 
 \end{proof}
 
The limit theorems associated with $\cLbulk$ require 
certain moments conditions on $\cDtpl(\bx,\cP_{\lambda})$, 
the scaled nearest-neighbour distance from $\bx \in \Omega_\lambda$ to its nearest point of
the Poisson process $\cP_{\lambda}$,
and on $\cDtpl(\bx,\cP_{\lambda} \cup \{\by\})$, where an additional point $\by \in \Omega_\lambda$ is inserted;
the required estimates are uniform over $\bx, \by$ and $\lambda$.
This is the purpose of the next result.

\begin{proposition}[Moments bounds]
\label{prop:moments-bounds}
Let $p >0$. Take $\Omega_\lambda$ as defined in~\eqref{eq:Omega-lambda}.
In the unaligned case, the moments bounds~\ref{prop:moments-bounds-a} and~\ref{prop:moments-bounds-b} hold.
In the singly-aligned case, there exists $c_{\phi,p} < \infty$ such that~\ref{prop:moments-bounds-a} and~\ref{prop:moments-bounds-b} hold for all $c$ in the definition of $\alpha_\lambda$ satisfying $c > c_{\phi,p}$. 
\begin{enumerate}[label=(\alph*)] % (a), (b), (c), ...
\item
\label{prop:moments-bounds-a}
$\sup_{\lambda \geq 1}\sup_{\bx \in \Omega_\lambda} \Exp[ (\cDtpl (\bx,\cP_{\lambda}) )^p ] < \infty$,
 \item
 \label{prop:moments-bounds-b}
 $\sup_{\lambda \geq 1}\sup_{\bx,\by \in \Omega_\lambda} \Exp[ (\cDtpl (\bx,\cP_{\lambda} \cup \{\by\}) )^p ] < \infty$.
\end{enumerate}
\end{proposition}

To 
prove Proposition~\ref{prop:moments-bounds}
we will show that $\cDtpl(\bx,\cP_{\lambda}) \leq R_{\lambda}(\bx)$ and $\cDtpl(\bx,\cP_{\lambda} \cup \{\by\}) \leq R_{\lambda}(\bx)$, and obtain both parts of the proposition from the following 
tail bounds on $R_{\lambda}(\bx)$.

 \begin{lemma}[Stabilization tail bounds]
 \label{lem:tau-lambda} 
Recall the definitions of $\Omega_\lambda$ and $\alpha_\lambda$ from around~\eqref{eq:Omega-lambda}.
Define also  $\tau_{\lambda}(r):= \sup_{\bx \in \Omega_{\lambda}}\Pr(R_{\lambda}(\bx) > r)$ for $\lambda >0$, $r \in \RP$.
  \begin{thmenumi}[label=(\roman*)]
        \item\label{lem:tau-lambda-singly-aligned}  
If $(\theta,\phi)$ is singly aligned, then there is a constant $b'_\phi \in (0,\infty)$ so that, for all $\lambda >0$, 
\begin{equation}
\label{eq:tau-lambda-singly-aligned}
   \tau_{\lambda}(r+1 )
\leq 
\begin{cases}
  \exp\{-b'_\phi r^{2}\}, & \text{if}\ 0 \leq r\leq \alpha_{\lambda}\sqrt{\lambda} ,\\
  \exp\{-\lambda b'_\phi \alpha_{\lambda}^{2}\}, & \text{if}\ \alpha_{\lambda}\sqrt{\lambda} \leq r \leq \sqrt{2\lambda}   ,\\
  0, & \text{if}\  r >   \sqrt{2\lambda}  .
\end{cases}
\end{equation}
\item 
\label{lem:tau-lambda-unaligned} 
If $(\theta,\phi)$ is unaligned, then there is a constant $b'_{\theta,\phi}   \in (0,\infty)$ so that, for all $\lambda >0$, 
    \begin{equation}
    \label{eq:tau-lambda-unaligned}
  \tau_{\lambda}(r +1 )
\leq 
\begin{cases}
  \exp\{-b'_{\theta,\phi}  r^{2}\}, & \text{if}\  0 \leq r \leq \sqrt{2\lambda}   ,  \\
  0, & \text{if}\  r >   \sqrt{2\lambda}  . 
\end{cases}
\end{equation}
\end{thmenumi}
 \end{lemma}

\begin{remark} 
\label{rem:tau-lambda} 
%In the terminology of Penrose (Definition~$2.3$~\cite{mp-gaussian}) $\cDtp$ is homogeneously stabilizing (\red{why is that?}). 
As mentioned above, we will show (see~\eqref{eq:D-by-R-bound} below) that $\cDtpl (\bx, \cP_{\lambda}) \leq R_{\lambda}(\bx)$. On the other hand, if $\cDtpl (\bx, \cP_{\lambda}) > 0$ then $R_{\lambda}(\bx)    \leq \cDtpl(\bx,\cP_{\lambda}) +1$ (see Lemma~\ref{lem:R-to-D-relations}), so the tails of $\cDtpl (\bx, \cP_{\lambda})$ and $R_{\lambda}(\bx)$ are essentially equivalent.
 \end{remark}

To prove Lemma~\ref{lem:tau-lambda}, we identify for each $s>0$ a region $A_{\theta, \phi}(\bx, s)$ of cone $\cCtp(\bx)$ such that $R_{\lambda}(\bx)$ is only bigger than $s$ if the region $A_{\theta, \phi}(\bx, s)$ contains none of $\cP_{\lambda}$. Then 
we bound
$\Pr(R_{\lambda}(\bx) \geq r)$ via lower bounds on the area $|A_{\theta, \phi}(\bx, s)|$. To define the appropriate region, recall from~\eqref{eq:ball-def} that $B_r(\bx)$ denotes the closed Euclidean ball centered at $\bx$ with radius~$r$.
 
 \begin{definition}
 \label{def:A-region}
For $\bx \in [0,1]^2$ and $s >0$, define
$     A_{\theta, \phi}(\bx, s) := \cCtp(\bx) \cap B_s(\bx) \cap [0,1]^2 
$.
\end{definition}

 Recall the definition of 
  $\nu_{\theta,\phi}(\bx)$
 from~\eqref{eq:nu-def}.
 For $\bx = (x_{1}, x_{2})$, let $\ell(\bx) := x_{1} \wedge x_{2}$.

\begin{lemma} 
\label{lem:A-lower-bound}
The following bounds on the area of $A_{\theta, \phi}(\bx, s)$ from Definition~\ref{def:A-region} hold for all $s>0$.
  \begin{thmenumi}[label=(\roman*)]
        \item\label{lem:A-lower-bound-i}
If $(\theta,\phi)$ is singly aligned and $\phi > \frac{\pi}{2}$, then there exists $b_\phi >0$ such that 
\[
\inf_{\bx \in \Omega_\lambda}| A_{\theta, \phi}(\bx, s)| \geq b_\phi (s \wedge \alpha_\lambda)^{2}.
\]
   \item\label{lem:A-lower-bound-ii}
If $(\theta,\phi)$ is singly aligned and $\phi < \frac{\pi}{2}$, then there exists $b_\phi >0$ such that 
    \begin{align}
        | A_{\theta, \phi}(\bx, s)| \geq b_\phi (s \wedge \ell(\bx))^{2},  \text{ for all }  \bx \in \Omega_{\lambda}. \nonumber
    \end{align}
      \item\label{lem:A-lower-bound-iii}
If $(\theta,\phi)$ is unaligned, then there exists $b_{\theta,\phi} >0$ such that 
    \begin{align}
        | A_{\theta, \phi}(\bx, s)| \geq b_{\theta,\phi} (s \wedge \nu_{\theta,\phi}(\bx))^{2},  \text{ for all }  \bx \in [0,1]^2 \nonumber
    \end{align}
\end{thmenumi}
\end{lemma}

To prove Lemma~\ref{lem:A-lower-bound}\ref{lem:A-lower-bound-iii}, we will need the following claim, whose proof we defer until later.% (\red{appendix}?).

\begin{lemma}
\label{lem:nadim}
    For $(\theta,\phi)$ unaligned, there exists $\eps := \eps(\theta,\phi) >0$ such that, for all $\bx \in [0,1]^2$, $ \cCtp(\bx) \cap [0,1]^2 $ contains a sector with apex $\bx$, angle $\eps$ and radius $\eps \nu_{\theta,\phi}(\bx)$.
\end{lemma}

\begin{proof}[Proof of Lemma~\ref{lem:A-lower-bound}]
To prove part~\ref{lem:A-lower-bound-i}, we show $A_{\theta, \phi}(\bx, s)$ always contains a sector with vertex $\bx$, angle $\phi - \frac{\pi}{2}$, and radius $s \wedge \alpha_{\lambda}$. First suppose that $\bx = (x_1, x_2) \in \Omega_\lambda$ is such that $0 \leq x_1 \leq 1 -\alpha_\lambda$. Then the south-east quadrant $B_{\mathrm{SE}}$ of the ball $B_{s \wedge \alpha_\lambda} (\bx)$ is contained in  $[0,1]^2$, implying $A_{\theta, \phi}(\bx, s) \supseteq A_{\theta, \phi}(\bx,s\wedge \alpha_{\lambda}) \cap B_{\mathrm{SE}} = \cCtp(\bx) \cap B_{\mathrm{SE}}$. On the other hand, if $1 -\alpha_\lambda \leq x_1 \leq 1$, a similar argument works but for the south-west quadrant of the ball (recall that $\alpha_\lambda < 1/2$).

For part~\ref{lem:A-lower-bound-ii}, observe that if $s \leq \ell(\bx)$, $A_{\theta, \phi}(\bx, s)$ equals the sector $\cCtp(\bx) \cap B_{s \wedge \ell (\bx)}(\bx)$ which has area $b_\phi (s\wedge \ell(\bx))^{2}$.

Finally we prove part~\ref{lem:A-lower-bound-iii}. Write $\nu := \nu_{\theta,\phi} (\bx)$ for the duration of this proof. 
 By Lemma~\ref{lem:nadim},  $ \cCtp(\bx) \cap [0,1]^2$ contains a sector with angle $\eps$ and radius $\eps \nu $. Then $A_{\theta, \phi}(\bx,s) = \cCtp(\bx) \cap [0,1]^2 \cap B_{s}(\bx)$ contains a sector of angle   $\eps$ and radius $s \wedge \eps \nu \geq \eps(s\wedge \nu)$, which has area $b_{\theta,\phi} (s \wedge \nu(\bx))^{2}$ (recall $\eps$ depended  on $\theta$ and $\phi$).
\end{proof}

\begin{proof}[Proof of Lemma~\ref{lem:tau-lambda}]
First note that, from Lemma~\ref{lem:R-to-D-relations}
and~\eqref{eq:nu-def},  $R_{\lambda}(\bx) \leq \lceil \lambda^{1/2} \nu_{\theta,\phi}(\bx) \rceil$.
Since $\nu_{\theta,\phi}(\bx) \leq \sqrt{2}$, this means that, in  particular,  $\Pr(R_{\lambda}(\bx) > r) = 0$ for $r \geq \lceil \sqrt{2\lambda} \rceil$.

The only candidate nearest neighbours for $\bx \in [0,1]^2$ within distance $\lambda^{-1/2} r$ are the points of
$\cP_{\lambda} \cap A_{\theta, \phi}(\bx, \lambda^{-1/2}r)$. If
there is at least one such point, then $\cDtp(\bx,\cP_{\lambda})$ takes value in $(0, \lambda^{-1/2}r]$, and 
 Lemma~\ref{lem:R-to-D-relations} shows
that $R_\lambda (\bx) = \lceil  \lambda^{1/2} \cDtp(\bx,\cP_{\lambda}) \rceil
\leq \lceil r \rceil < r +1$.
%, i.e., $R_\lambda (\bx) \leq r$ (since  $R_\lambda (\bx)$  is an integer).
Hence, for every $\bx \in [0,1]^2$ and all $r \in \RP$, 
\begin{align} 
    \Pr(R_{\lambda}(\bx) > r+1)
    \leq \Pr(      \cP_{\lambda} \cap A_{\theta, \phi}(\bx, \lambda^{-1/2}r) = \emptyset)
    = \exp\{-\lambda|A_{\theta, \phi}(\bx, \lambda^{-1/2}r)|\} \label{eq:d1}.
\end{align}

To prove part~\ref{lem:tau-lambda-singly-aligned} of the lemma,
suppose that $(\theta,\phi)$ is singly aligned. First consider the obtuse case, $\phi > \frac{\pi}{2}$. Then Lemma~\ref{lem:A-lower-bound}\ref{lem:A-lower-bound-i} and inequality~\eqref{eq:d1}  give
\begin{align}
    \sup_{\bx \in \Omega_{\lambda}} \Pr(R_{\lambda}(\bx) > r+1) \leq  \exp \Bigl\{ -\lambda\inf_{\bx \in \Omega_{\lambda}}|A_{\theta, \phi}(\bx, \lambda^{-1/2}r)| \Bigr\}
    \leq \exp\{-\lambda b_\phi (\lambda^{-1/2}r \wedge \alpha_\lambda)^{2}\} \nonumber.
\end{align}
This upper bound matches (with $b'_\phi = b_\phi >0$)
that claimed in~\eqref{eq:tau-lambda-singly-aligned} for the cases $r \leq \alpha_{\lambda} \lambda^{1/2}$ and $r \geq \alpha_{\lambda} \lambda^{1/2}$.
Next consider the acute case, $\phi < \frac{\pi}{2}$. 
Here,   by Lemma~\ref{lem:A-lower-bound}\ref{lem:A-lower-bound-ii}  and~\eqref{eq:d1}, 
\begin{align}
 \Pr(R_{\lambda}(\bx) > r+1) \leq  \exp\{-\lambda|A_{\theta, \phi}(\bx, \lambda^{-1/2}r)|\}
    \leq \exp\{-\lambda b_\phi( (\lambda^{-1/2}r) \wedge \ell(\bx))^{2}\}. \nonumber
\end{align}
Recalling that $R_{\lambda}(\bx) \leq \lceil \lambda^{1/2}\nu_{\theta,\phi}(\bx) \rceil$, we have that $\Pr(R_{\lambda}(\bx) > r+1) = 0$ if $r \geq \lambda^{1/2} \nu_{\theta,\phi}(\bx)$.  Hence, for fixed $r$, defining $\Omega_\lambda(r) := \{ x \in \Omega_\lambda : \lambda^{1/2} \nu_{\theta,\phi}(\bx)  > r \}$,
\[
    \sup_{\bx \in \Omega_{\lambda}}\Pr(R_{\lambda}(\bx) > r+1)   = \sup_{\bx \in \Omega_{\lambda}(r)}\Pr(R_{\lambda}(\bx) > r+1)  
      \leq  \exp \Bigl\{ -\lambda \inf_{\bx \in \Omega_{\lambda}(r) } b_\phi (\lambda^{-1/2}r \wedge \ell(\bx))^{2} \Bigr\}. 
\]

For all $\bx = (x_1,x_2) \in \Omega_\lambda = [0,1] \times [\alpha_\lambda,1]$, we have that if $\ell(\bx) < \alpha_{\lambda}$ then $x_{1}\leq \alpha_{\lambda}$ and $\nu_{\theta,\phi}(\bx) \leq \ell(\bx) \sec\phi$, i.e., $\ell(\bx) \geq \nu_{\theta,\phi}(\bx) \cos \phi $.
Therefore the infimum of $\ell(\bx)$ over $\Omega_\lambda(r)$ is at least $\alpha_{\lambda} \wedge ( \lambda^{-1/2} r \cos \phi  )$. 
Hence
\begin{align}
     \sup_{\bx \in \Omega_{\lambda}}\Pr(R_{\lambda}(\bx) > r+1)
     \leq \exp\{ -\lambda b_\phi (\alpha_{\lambda} \wedge  \lambda^{-1/2} r\cos \phi )^{2} \}
     \leq  \exp\{ -\lambda ( b_\phi  \cos^2 \phi )  (\alpha_{\lambda} \wedge  \lambda^{-1/2} r)^{2} \}
     . \nonumber
\end{align}
This verifies the upper bound claimed in~\eqref{eq:tau-lambda-singly-aligned} for $b'_\phi = b_\phi \cos^2 \phi >0$.

Finally, for part~\ref{lem:tau-lambda-unaligned}, suppose that
$(\theta,\phi)$ is unaligned.  In this case, similarly to the previous one, but using Lemma~\ref{lem:A-lower-bound}\ref{lem:A-lower-bound-iii}, we have
\[
    \sup_{\bx \in \Omega_{\lambda}}\Pr(R_{\lambda}(\bx) > r+1) 
    %&= \sup_{\bx \in \Omega_{\lambda}(r) } \Pr(R_{\lambda}(\bx) > r) 
    \leq \exp\Bigl\{ -\lambda \inf_{\bx \in \Omega_{\lambda}(r)} b_{\theta,\phi} (\lambda^{-1/2}r \wedge \nu_{\theta,\phi}(\bx))^{2} \Bigr\} = \exp\{ -\lambda b_{\theta,\phi} (\lambda^{-1/2} r)^{2}\},
\]
which verifies~\eqref{eq:tau-lambda-unaligned} with $b'_{\theta,\phi} = b_{\theta,\phi}>0$.
\end{proof}

\begin{proof}[Proof of Lemma~\ref{lem:nadim}]
Let $\bx \in [0,1]^2$ and $A(\bx) := \cCtp(\bx) \cap [0,1]^2$.  Throughout the proof we write $\nu(\bx):= \nu_{\theta,\phi}(\bx)$.  If $A(\bx) = \{ \bx \}$ (which may happen when neither of the boundary rays of $\cCtp(\bx)$ intersects the interior of $[0,1]^2$) then $\nu(\bx) = 0$ and the statement is trivially true.  In all other cases, since $(\theta, \phi)$ are unaligned, the set $A(\bx)$ is a non-trivial polygon whose vertices comprise $\bx \in [0,1]^2$ and a finite number of points lying on the boundary of $[0,1]^2$.  Note, by definition of $\nu(\bx)$, and since the function $\by \mapsto \|\bx - \by\|$ is continuous and convex, its maximum over $\by \in A(\bx)$ (equal to $\nu(\bx)$) is attained at an extreme point of the convex set $A(\bx)$ 
 (see e.g.~\cite{MR2335496} for maxima of convex functions).

Suppose for the rest of the proof that  $A(\bx) \neq \{ \bx \}$. It suffices to prove that there is $c_{\theta,\phi} >0$ such that $| A (\bx) | \geq c_{\theta,\phi} \nu(\bx)^2$ for all $\bx$ with $\nu(\bx)>0$. Indeed,
$A (\bx)$ is a polygon with at most 6 vertices, so triangulating it from vertex $\bx$ gives a triangle inside $A(\bx)$ with area at least $| A (\bx )| /4$.  Then, writing $a,b$ for the lengths of the sides of that triangle incident to $\bx$ and $\alpha \in (0,\pi)$ for the angle at $\bx$, the bounds $\max\{a,b\} \leq \nu(\bx)$ and $\sin \alpha \leq 1$ imply
\[
\tfrac12\min\{a,b\} \nu(\bx) \geq \tfrac12ab\sin\alpha \geq \tfrac14c_{\theta,\phi} \nu(\bx)^2~~~\text{and}~~~\tfrac12\nu(\bx)^2\sin\alpha \geq \tfrac12ab\sin\alpha \geq \tfrac14c_{\theta,\phi} \nu(\bx)^2,
\]
so that both $\min\{a,b\}/\nu(\bx)$ and $\alpha$ have positive lower bounds (independent of $\bx$).  Such a triangle necessarily contains a sector as claimed in the statement of the lemma.

Suppose, in order to derive a contradiction, that there is no such $c_{\theta,\phi} >0$. Then there exist $\bx_1, \bx_2, \ldots \in [0,1]^2$, with $\nu (\bx_n)>0$ for all $n$, such that $|A(\bx_n)| /\nu(\bx_n)^2 \to 0$ as $n \to \infty$.
Suppose there is a subsequence of $\bx_n$ along which $\nu (\bx_n)$ is bounded away from zero. Then, by passing to a further subsequence, $\bx_n \to \bx \in [0,1]^2$, and, by continuity, $| A(\bx_n) | / \nu(\bx_n)^2 \to | A (\bx) | / \nu(\bx)^2$. But if $\nu (\bx) >0$ then $A(\bx)$ must have positive area (since the cone is not aligned). This contradicts the choice of $\bx_n$, meaning that we must have $\nu (\bx_n) \to 0$ as $n \to \infty$.

Consider the set $B_n := ( A (\bx_n) - \bx_n )/ \nu (\bx_n) = \cCtp (\0) \cap K_n$
where $K_n := ([0,1]^2 - \bx_n)/\nu (\bx_n)$. 
After taking a subsequence, each of the four scaled perpendicular distances from $\0$ to the sides of the square $K_n$ converges in $[0,\infty]$. Since opposite scaled distances sum to $1/\nu(\bx_n)\to\infty$, in each coordinate at most one limiting boundary is finite. Hence $K_n$ converges locally to $K_\infty$, where $K_\infty$ is either the plane, a horizontal/vertical half-plane, or the intersection of one horizontal and one vertical half-plane.
Now $B_n$ is contained in the closed ball of unit radius, by the definition of $\nu(\bx_n)$, and $B_n$ contains a point of norm $1$. Passing to a subsequence, $B_n \to B_\infty = \cCtp (\0) \cap K_\infty$
in Hausdorff distance, with $B_\infty$ non-trivial. If $B_\infty$ had zero area, this intersection would have to degenerate to a line segment or ray. But that can occur only when a boundary ray of $\cCtp (\0)$ coincides with a horizontal or vertical boundary direction of $K_\infty$. That is  excluded by the unaligned hypothesis. Therefore $|B_\infty|>0$.
On the other hand, continuity of area under Hausdorff convergence of compact convex sets gives
\[
|B_\infty|
=\lim_{n\to\infty} |B_n|
=\lim_{n\to\infty} \frac{|A(\bx_n)|}{\nu(\bx_n)^2}
=0,
\]
which is a contradiction.
\end{proof}

Now we can complete the proof of Proposition~\ref{prop:moments-bounds}.
 
 \begin{proof}[Proof of Proposition~\ref{prop:moments-bounds}] For all $\bx \in \Omega_{\lambda}$ and for all $\lambda \geq 1$, by Lemma~\ref{lem:R-to-D-relations}, we have
 \begin{align}
      \cDtpl(\bx,\cP_{\lambda}) =  \lambda^{\frac{1}{2}} \cDtp(\bx,\cP_{\lambda}) \leq \lceil  \lambda^{\frac{1}{2}} \cDtp(\bx,\cP_{\lambda}) \rceil \leq R_{\lambda}(\bx). 
      \label{eq:D-by-R-bound}
 \end{align}
Similarly, using Lemma~\ref{lem:R-monotonicity}, 
we have for all $\bx \in \Omega_{\lambda}$,
all $\by \in \Omega_{\lambda}$,
and all $\lambda \geq 1$, that
 \begin{align}
     \cDtpl(\bx,\cP_{\lambda} \cup \{\by\})  \leq \lceil  \lambda^{\frac{1}{2}} \cDtp (\bx,\cP_{\lambda}\cup \{\by\}) \rceil  &\leq R(\lambda^{\frac{1}{2}}\bx, \lambda^{\frac{1}{2}} (\cP_{\lambda}\cup \{\by\}), [0,\lambda^{\frac{1}{2}}]^{2}) \nonumber \\
    & \leq R(\lambda^{\frac{1}{2}}\bx, \lambda^{\frac{1}{2}}\cP_{\lambda}, [0,\lambda^{\frac{1}{2}}]^{2}) = R_{\lambda}(\bx). \nonumber
 \end{align}
 So to prove both statements~\ref{prop:moments-bounds-a} and~\ref{prop:moments-bounds-b}, it suffices to prove that $\sup_{\lambda \geq 1}\sup_{\bx \in \Omega_{\lambda}} \Exp[R_{\lambda}(\bx)^{p}] < \infty$. 

 First suppose that $(\theta,\phi)$ is singly aligned.
For all $p>0$, using the bounds in Lemma~\ref{lem:tau-lambda}\ref{lem:tau-lambda-singly-aligned}, 
\begin{align*} 
    \sup_{\bx \in \Omega_\lambda} \Exp[R_{\lambda}(\bx)^{p}] \leq 1 + \int_{1}^{\infty} \tau_\lambda (s^{1/p} ) \ud s 
    & \leq 1+ \int_1^\infty \re^{-b_\phi' (s^{1/p}-1)^2 } \ud s + 
    \int_{\alpha^p_\lambda \lambda^{p/2}}^{(1+\sqrt{2\lambda})^{p}} \re^{-\lambda b_\phi' \alpha_\lambda^2}  \ud s \\
    & \leq C_{\phi,p} + (1+\sqrt{2\lambda})^{p} \re^{-\lambda b_\phi' \alpha_\lambda^2} ,
\end{align*}
where $C_{\phi,p} < \infty$. 
Here, provided $c \geq c_{\phi,p}:= p/(2b'_\phi)$,
$\sup_{\lambda \geq 1}  (1+\sqrt{2\lambda})^{p} \re^{-\lambda b_\phi' \alpha_\lambda^2} < \infty$,
since for all $\lambda$ large enough we have $\alpha_\lambda \geq \sqrt{ c \log \lambda /\lambda }$ so that 
$\re^{-\lambda b_\phi' \alpha_\lambda^2} \leq \lambda^{-cb_\phi'}$. If $(\theta,\phi)$ is unaligned, then
by Lemma~\ref{lem:tau-lambda}\ref{lem:tau-lambda-unaligned},
\begin{align*} 
    \sup_{\bx \in \Omega_\lambda} \Exp[R_{\lambda}(\bx)^{p}] \leq 1 + \int_{1}^{\infty} \tau_\lambda (s^{1/p} ) \ud s 
    & \leq 1 + \int_1^\infty \re^{-b_{\theta,\phi}' (s^{1/p}-1)^2 } \ud s \leq C_{\theta,\phi,p}  < \infty ,
\end{align*}
for every $p >0$.
This completes the proof.
\end{proof}

Recall from Lemma~\ref{lem:tau-lambda} that 
$\tau_{\lambda}(s) =  \sup_{\bx \in \Omega_{\lambda}}\Pr(R_{\lambda}(\bx) > s)$. Then we further define,
for $s>0$, 
$\tau(s) := \sup_{\lambda \geq 1}\tau_{\lambda}(s)$.
The following result shows \emph{power law stabilization} for the MDSF total edge length functional.

\begin{lemma} 
\label{lem:tau-moments}
\begin{thmenumi}[label=(\roman*)]
        \item\label{lem:tau-moments-i}
        If $(\theta,\phi)$ is singly aligned, then for 
        every $q \in (0, 2 b'_\phi c)$, where $b'_\phi$ is as in Lemma~\ref{lem:tau-lambda}\ref{lem:tau-lambda-singly-aligned} and $c>0$ is as in the definition of $\alpha_\lambda$, 
$\sup_{s \geq 1}s^{q}\tau(s) < \infty$.
\item\label{lem:tau-moments-ii} If $(\theta,\phi)$ is unaligned, 
then $\sup_{s \geq 1}s^{q}\tau(s) < \infty$ for every $q>0$.
\end{thmenumi}
\end{lemma}
\begin{proof} 
In each case, it is enough to show that $\lim_{s \to \infty} s^q \tau (s) = 0$.
For part~\ref{lem:tau-moments-i}, we have from 
Lemma~\ref{lem:tau-lambda}\ref{lem:tau-lambda-singly-aligned}
that $\tau_\lambda (s+1) \leq \exp ( - b_\phi' ( s^2 \wedge \lambda \alpha_\lambda^2))$ for $s \leq \sqrt{2\lambda}$ and $\tau_\lambda (s+1) =0$ for $s > \sqrt{2 \lambda}$. Hence $\tau(s+1) \leq \sup_{\lambda \geq s^2/2} \exp ( -b_\phi' (s^2 \wedge \lambda \alpha_\lambda^2 ) )$.
Thus for all $s$ sufficiently large,
\[ \tau (s+1) \leq \sup_{\lambda \geq s^2/2} \exp ( -b_\phi' (s^2 \wedge c \log \lambda ) )
= \exp ( -b_\phi' (s^2 \wedge c \log (s^2/2) ) )
= (s^2/2)^{-b_\phi' c} .\]
So, as required,  $\lim_{s \to \infty} s^q \tau (s) = 0$
provided $q < 2 b_\phi' c$, as claimed in part~\ref{lem:tau-moments-i}.
Part~\ref{lem:tau-moments-ii} follows immediately from Lemma~\ref{lem:tau-lambda}\ref{lem:tau-lambda-unaligned}, which shows 
$\tau(s+1) = \sup_{\lambda \geq 1}\tau_{\lambda}(s+1) \leq \exp (- b'_{\theta,\phi}  s^{2} )$.
\end{proof}

\subsection{Proof of central limit theorem and variance asymptotics}
\label{sec:clt-proof}

We are now ready to prove Theorem~\ref{thm:bulk-clt}, which essentially follows from an application of Theorems~$2.1$ and~$2.2$ of~\cite{mp-gaussian}, which give a central limit theorem and variance asymptotics for general stabilising functionals.   Before doing so, we claim there is a positive lower bound on the variance of $\cLbulk$, which justifies that the constant $s_\phi$ appearing in  Theorem~\ref{thm:bulk-clt} is positive.
    \begin{proposition} 
    \label{prp:proof1.5.4.} 
    There exists $v_0 > 0$ (depending on $\phi$), such that, for all $\lambda$ sufficiently large, 
    \begin{equation}
\Var ( \cLbulk ) \geq v_0. \nonumber
    \end{equation}
\end{proposition}
We omit the proof of Proposition~\ref{prp:proof1.5.4.} here, as the idea is standard, and goes 
back at least to Proposition~5 of~\cite{avram-bertsimas}: roughly speaking, 
one shows that in a region where one expects to find $O(1)$ Poisson points, there is positive
probability of a `local' configuration that isolates a small number of Poisson points from the rest, and contributes positive variance. See  \cite{pw-south-west,pw-south} for such arguments in the MDSF with rectangular cones of directions, and Theorem~2.2 of~\cite{pw-multivariate} for a conveniently applicable general statement.

\begin{proof}[Proof of Theorem~\ref{thm:bulk-clt}]
We apply Theorems~$2.1$ and~$2.2$ of~\cite{mp-gaussian} to the functional $\cDtp$, which satisfies the required hypotheses, by the results from Section~\ref{sec:stabilization}.  Indeed, translation invariance is given by \eqref{eq:translation-invariant} and, since $\phi > 0$, the cone $\cCtp(\bx)$ almost surely contains a point of a homogeneous Poisson process on $\R^2$ at finite distance from $\bx$ and hence the homogeneous stabilisation condition (2.10) of \cite{mp-gaussian} is satisfied.  Proposition~\ref{prop:moments-bounds}
 shows that for $p=4$ there exists $c_{\phi,p} < \infty$ such that the required $p$-th moments conditions hold for all $c > c_{\phi,p}$. Lemma~\ref{lem:tau-moments}
shows that we may choose this $c$ large enough such that the required power-law stabilization condition holds with $q=304 > 2(150+6/p)$.  
This verifies the hypotheses of Theorems~$2.1$ and~$2.2$ of~\cite{mp-gaussian} (note also the choice of height $R S$ means that $\cRbulk$ are non-decreasing in $\lambda$) and we conclude that
\begin{align*}
    \Var(\cLbulk) \to s_{\phi}^{2}~~~\text{ and}~~~
 \tcLbulk \tod \cN(0,s^2_\phi)
\end{align*}
 as  $\lambda \to \infty$ for some $s_{\phi} < \infty$. By Proposition~\ref{prp:proof1.5.4.}, the limit $s_{\phi}^{2}$ is greater than or equal to $v_0 > 0$, and therefore positive.

Observe that, since $s_{\phi}$ can be expressed in terms of functionals of a homogeneous Poisson process on $\R^2$, which is rotationally invariant (see~\cite{mp-gaussian} for more details), it depends only on $\phi$ and not $\theta$.
\end{proof}

\section{Boundary effects for singly-aligned cones}
\label{sec:boundary-effects}

\subsection{Introduction}
We now turn our attention to singly-aligned cones and the proof of Theorem~\ref{thm:generalresult}\ref{thm:generalresult-singly-aligned}.  In this case, there are boundary effects which contribute a non-Gaussian component to the limit distribution, coming from edges near to the boundary of the unit square.  In this section we show that this boundary contribution converges in distribution to a random variable with distribution $\cQ$ as characterised by \eqref{eq:fixed-point}.
The results in this section are analogues of those for the doubly-aligned case from~\cite[\S 6]{pw-south-west}.

Recall from Section~\ref{sec:partioning-the-square} that for a singly-aligned cone we define a non-empty boundary region $\cRbndry=[0,1] \times [0,\lambda^{-\sigma}]$, for $\sigma \in (\frac{1}{2},\frac{2}{3})$; this is where we expect to see long edges. As in that section, we write $\cLbndry = \cL(\cP_\lambda,\cRbndry)$ for the contribution to the total length from points in $\cP_\lambda \cap \cRbndry$
 and $\tcLbndry = \cLbndry - \Exp \cLbndry$ for the centered random variable. The principal result of this section is the following convergence in distribution of $\tcLbndry$.
 
\begin{theorem} 
\label{thm:boundary-effects}
Let $\sigma \in (\frac{1}{2}, \frac{2}{3})$. 
If $(\theta,\phi)$ is singly-aligned, then
\begin{equation}
    \tcLbndry \tod  Q,  \text{ as } \lambda \to \infty\label{eq:boundary-effects},
\end{equation}
where $Q \sim \cQ$ and the distribution $\cQ$ is characterized by the fixed point equation~\eqref{eq:fixed-point}.
\end{theorem}

To prove Theorem~\ref{thm:boundary-effects} we construct a \emph{directed linear forest} (DLF) defined on a sequence of random variables that are coupled to the points of the Poisson process $\cP_\lambda$, and show that total edge length of MDSF on points near the bottom boundary is asymptotically close to the total edge length of this DLF.
The DLF was introduced by Penrose and Wade \cite{pw-south-west}; we summarise what will be relevant for the present paper in Section~\ref{sec:DLF}.

\subsection{Directed linear forest}
\label{sec:DLF}

In the DLF, every point in a sequence of independent uniform random points within an interval is connected to its closest left neighbour among the preceding points;  see~\cite{pw-south-west} for further details. 
Let $\cS$ be a finite sequence in $(0,1)$, i.e., $\cS = (S_1, \ldots, S_m) \in (0,1)^m$.
 DLF($\cS$), 
is constructed by, for every $i \in \{2,3,\ldots,m\}$, inserting a directed edge
from $S_i$ to the closest $S_j$ with $1 \leq j<i$ and $S_j<S_i$, i.e.,
every point after the first in the sequence is joined by an edge to its
nearest left-neighbour, if there is one. Denote the total edge length of DLF($\cS$) by
\begin{equation}
    \label{eq:DLF-length}
 \cLdlf{ {S_1, \ldots, S_m} } := \sum_{i=2}^m H_i , \text{ where } H_i :=   S_i - \max \{ S_j : 1 \leq j < i, S_j < S_i \}   ,
\end{equation}
whenever the set inside the $\max$ is non-empty, and $H_i:= 0$ if the set inside the $\max$ is empty. 
We recall Theorem~3.1 of Penrose and Wade~\cite{pw-south-west} below for the total length of the directed linear forest on uniform random points in the interval.

\begin{theorem}[Penrose and Wade~\cite{pw-south-west}]
\label{thm:PW-DLF}
For $n \in \N$, let $\cU_n := (U_1, U_2, \ldots, U_n)$ where $U_1, \ldots, U_n$ are i.i.d.~$\unif{0}{1}$. 
Then, as $n \to \infty$, the total edge length of DLF($\cU_n$) satisfies
    \[    \cLdlf{\cU_n} - \Exp \cLdlf{\cU_n} \tod  Q, \]
    where $Q \sim \cQ$ has distribution characterized by the fixed-point equation~\eqref{eq:fixed-point}.
\end{theorem}

To prove Theorem~\ref{thm:boundary-effects}, we couple the (total length of the) MDSF on the two-dimensional set of points in the boundary region to  the (total length of the) DLF defined on the horizontal components of those points, ordered by increasing vertical component.
The notation for this is as follows.

\begin{definition}[Boundary sequences]
Let $\cW_\lambda := \cP_\lambda \cap  \cRbndry$ and $M_{\lambda} := \card {\cW_\lambda}$.
Enumerate points of $\cW_\lambda$ as $\cW_\lambda = (Z_i = (X_i,Y_i),  1 \leq i \leq M_{\lambda})$ by order of increasing vertical coordinate: $Y_1< Y_2 < \cdots < Y_{M_{\lambda}}$.
Then set $\cV_\lambda := (X_i, i = 1,2, \dots, M_{\lambda})$. 
\end{definition}

\begin{theorem} 
\label{thm:DLF-singly-aligned-approx}
Suppose that $(\theta,\phi)$ is singly-aligned. 
Let $\sigma \in (\frac{1}{2}, \frac{2}{3})$. 
Then,
\begin{equation}
  \cLbndry - \cLdlf{\cV_\lambda} \toL{2} 0,  \text{ as } \lambda \to \infty\label{eqn:1.7.3}.
\end{equation}
\end{theorem}

Before proving Theorem~\ref{thm:DLF-singly-aligned-approx}, we first write $\cLbndry$ and $\cLdlf{\cV_\lambda}$ as sums of random variables associated with points in $\cV_{\lambda}$. For this we require the following definitions.
\begin{definition}
\label{def:DLF-J}
For $i \in \{1, \dots ,M_\lambda\}$, let $J(i) := \max \{j <i : X_j = \max_{k<i}\{X_k: X_k < X_i\} \}$ denote the largest index of a rightmost predecessor of $i$, arriving to the left of $X_i$ (with the convention that $J(i) := \infty$ if there is none).
 \end{definition}

\begin{definition}
\label{def:DLF-K}
If $\cW_\lambda \cap \cCotp(Z_i) \neq \emptyset$, let $K(i) \in \{1, \dots ,M_\lambda\}$ be the unique index such that $Z_{K(i)} = N_{\theta,\phi}(Z_i; \cW_\lambda)$; otherwise let $K(i) := \infty$.
\end{definition}
Note that $K(i) < i$ whenever $K(i)$ is finite.
Using Definition~\ref{def:DLF-J}, we can write 
\begin{align}\label{L_3region} 
    \cLdlf{\cV_\lambda} = \sum_{i=1}^{M_{\lambda}}|X_i-X_{J(i)}|\ind_{\{J(i) < \infty\}}, 
\end{align}
(set $X_\infty :=0$ for definiteness, and similarly $Z_\infty := \0$) 
and using Definition~\ref{def:DLF-K}, we can write 
\begin{align}
     \cLbndry = \sum_{i=1}^{M_{\lambda}}\|Z_i-Z_{K(i)}\|\ind_{\{K(i) < \infty\}}.\label{def:L3region}
\end{align}
We then decompose $\cLbndry  - \cLdlf{\cV_\lambda} = \sum_{i=1}^{M_\lambda} \Delta_i$ as the sum  of contributions 
\begin{equation}
    \label{eq:Delta-i-def}
\Delta_i :=  \|Z_i -Z_{K(i)}\|\ind_{\{K(i) < \infty\}} - |X_i - X_{J(i)}|\ind_{\{J(i) < \infty\}}.
\end{equation}

Because the region $\cRbndry$ is short and wide, the (2-dimensional) distance $\|Z_i -Z_j\|$ between points with indices $i$ and $j$ is very close to their (1-dimensional) distance $|X_i-X_j|$.  Therefore, if the index set for the candidate neighbours of $Z_i$ in the MDSF is the same as that of $X_i$ in the DLF, then the corresponding lengths in the two constructions will be very close.   A substantial discrepancy in the lengths may arise when some indices are available as nearest neighbours in one construction but not the other.  For the acute cone, there are potential 1-dimensional neighbours that are not available as 2-dimensional neighbours;  for the obtuse cone, it is the other way around. 
%({\bf Figure of the triangles $\Eo_i$})  
However, the events that produce large discrepancies occur with small probability, so the $L^2$ convergence can still be controlled.  Because of the differing geometries for the two cones, the arguments required to prove   Theorem~\ref{thm:DLF-singly-aligned-approx} are slightly different; Section~\ref{sec:coupling-obtuse} covers the obtuse cone and Section~\ref{sec:coupling-acute} covers the acute cone.

\subsection{Coupling for the obtuse case}
\label{sec:coupling-obtuse}

In this section we assume that $\phi \in (\frac{\pi}{2}, \pi)$.
We begin by defining a region $\Eo_i$ of $\cRbndry$ for each point $Z_i$ of $\cW_\lambda$ which plays an important role in controlling the size of $\Delta_{i}$.  Indeed, any Poisson point in $\Eo_i$ is a candidate neighbour of $Z_i$ in the (2-dimensional) MDSF, but does not generate a candidate neighbour of $X_i$ in the associated DLF.

  \begin{definition} \label{Def:proof1.7.7.}
    For each $i = 1,2, \dots, M_{\lambda}$, define
  \begin{align}
      \Eo_{i} := \{(x,y) \in \cRbndry : x> X_i, (x,y) \in \cCtp(Z_i)
      \}. \nonumber
  \end{align}
  \end{definition}

Note that $\Eo_i$ is always contained in a right-angled triangle with height $\lambda^{-\sigma}$ and width $ \lambda^{-\sigma} \tan (\phi - \frac{\pi}{2})$.
For $i \in \N$ we define 
\begin{equation}
    \label{eq:A-i-def}
A_{i} :=\{ i \leq M_\lambda, \, J(i) < \infty, \, K(i) < \infty, \, Z_{K(i)} \in \Eo_i \},
\end{equation}
 and set 
 \begin{equation}
     \label{eq:S-primes-def}
 S_\lambda^{\prime} = \sum_{i=1}^{M_{\lambda}}\Delta_{i}\ind_{A^{\rc}_{i}}, \text{ and } S_\lambda^{\dprime} =  \sum_{i=1}^{M_{\lambda}}\Delta_{i}\ind_{A_{i}}. 
  \end{equation}
 We will show that both $S'_\lambda$ and $S''_\lambda$ converge to 0 in $L^2$.

\begin{lemma}
\label{lem:obtuse-boundary-error-1}
Suppose that $\phi \in (\frac{\pi}{2}, \pi)$ (obtuse case). Let $\sigma \in (\frac{1}{2}, \frac{2}{3})$. Then
\[
    S_\lambda^{\prime} \toL{2} 0,  \text{ as } \lambda \to \infty.\nonumber \]
\end{lemma}
\begin{proof}
We determine the value of $\Delta_i$ for each $i \in \{ 1, \dots, M_\lambda \}$, as follows.  Given that $A_{i}^{\rc}$ occurs, since $i \leq M_\lambda$ exactly one of the following events 
occurs:
\begin{enumerate}[label={\upshape(\roman*)}, align=left, widest=iii, leftmargin=*]
     \item \label{lem:obtuse-boundary-error-1-i} \{$K(i) = \infty\}$,
     \item \label{lem:obtuse-boundary-error-1-ii} \{$J(i) = \infty, K(i) < \infty\}$,
      \item \label{lem:obtuse-boundary-error-1-iii} \{$J(i) < \infty, K(i) < \infty, Z_{K(i)} \notin \Eo_i\}$.
\end{enumerate}
In all these cases,  geometric arguments imply that $\Delta_i$ is small. Indeed, in Case~\ref{lem:obtuse-boundary-error-1-i}, $K(i) =  \infty$ means $\cCotp(Z_i)$ contains no points of $\cW_{\lambda}$, so $J(i)=  \infty$ also. Then $\Delta_{i}= 0$.
In Case~\ref{lem:obtuse-boundary-error-1-ii}, $J(i) = \infty$ and $K(i) < \infty$, so that 
$\cCotp(Z_i)$ contains $Z_{K(i)}$ the nearest-neighbour of $Z_{i}$ (hence $K(i) < i$) but $X_{K(i)} > X_{i}$ (otherwise $J(i) < \infty$). Hence $ Z_{K(i)} \in \Eo_i$ and 
\begin{equation}
    0  \leq \Delta_i = \|Z_i-Z_{K(i)}\| \leq  \lambda^{-\sigma}\sec\left(\phi - \frac{\pi}{2}\right) \label{eqt:eq},
\end{equation}
since the length $\|Z_i-Z_{K(i)}\|$ is at most the length of the hypotenuse of $\Eo_i$.

In the remaining case, 
Case~\ref{lem:obtuse-boundary-error-1-iii},
$J(i) < \infty$, $K(i) < \infty$, and  
$ Z_{K(i)} \notin \Eo_i$ means $K(i) < i$ and $X_{K(i)}< X_{i}$, hence 
$|X_{J(i)} - X_{i}| \leq |X_{K(i)} - X_{i}|$
by definition of $J(i)$. 
Also, $Z_{J(i)} \in \cCotp(Z_i)$ so 
$\|Z_{K(i)}-Z_i\| \leq \|Z_{J(i)}-Z_i\|$ by definition of $K(i)$. Therefore
\begin{align*}
  |X_{J(i)}-X_i| \leq |X_{K(i)}-X_i| \leq \|Z_{K(i)}-Z_i\| &\leq  \|Z_{J(i)}-Z_i\| \\
  &\leq  |X_{J(i)}-X_i|+|Y_{J(i)}-Y_i|  \\
  &\leq |X_{J(i)}-X_i| + \lambda^{-\sigma}. 
 \end{align*}
Therefore, $\Delta_{i} = \|Z_{K(i)}-Z_i\| - |X_{J(i)}-X_i|$ satisfies $0 \leq \Delta_{i} \leq \lambda^{-\sigma}$.

Combining Cases~\ref{lem:obtuse-boundary-error-1-i},~\ref{lem:obtuse-boundary-error-1-ii}, and~\ref{lem:obtuse-boundary-error-1-iii}, yields 
\begin{align*}
  0 \leq \Delta_i\ind_{A^{\rc}_{i}} \leq  \lambda^{-\sigma}\sec \left(\phi - \frac{\pi}{2}\right),
  \text{ for all } i \in \{1,\ldots, M_\lambda\},
\end{align*}
so that $|S_\lambda^{\prime}| \leq M_{\lambda}\lambda^{-\sigma} \sec\left(\phi - \frac{\pi}{2}\right)$. Thus $\Exp[( S_\lambda^{\prime})^2] \leq  \lambda^{-2\sigma} \sec^2 \left(\phi - \frac{\pi}{2}\right)  \Exp[M_{\lambda}^2]$.
Since  $M_{\lambda} \sim \Poi(\lambda^{1-\sigma})$, we bound $\Exp[M_\lambda^2] = \lambda^{1-\sigma} + \lambda^{2(1-\sigma)} \leq 2\lambda^{2-2\sigma}$, for $\lambda > 1$.  Hence, since $\sigma \in (\frac{1}{2}, \frac{2}{3})$,
\begin{equation*}
     \Exp[( S_\lambda^{\prime})^2] \leq 2\lambda^{2-2\sigma}\cdot\lambda^{-2\sigma} \sec^{2}\left(\phi - \frac{\pi}{2}\right) \to 0,   \text{ as }  \lambda \to \infty . \qedhere
\end{equation*}
\end{proof}

We now turn our attention to the second sum $S_\lambda^{\dprime} = \sum_{i=1}^{M_{\lambda}}\Delta_{i}\ind_{A_{i}}$ from~\eqref{eq:S-primes-def}.
In this case, a large $\Delta_i$ can arise from
a small two-dimensional length and a large one-dimensional length. An $L^2$ bound can be achieved by combining the fact that~$A_i$ is itself unlikely (as $\Eo_i$ is small),
and that the one-dimensional length can also be controlled in terms of \emph{spacings} induced by the $X_{1}, X_{2}, \dots, X_{i}$. We make use of
the concentration of a Poisson random variable around its mean. Specifically,  Lemma $1.2$ from~\cite{penrose-book} (for example) states that if $X  \sim \Poi(\mu)$, $\mu > 0$, then
\begin{align}\label{eq:poisson-3-mu}
    \Pr ( X \geq 3\mu ) \leq \re^{-\mu}.
\end{align}

\begin{lemma}
\label{lem:obtuse-boundary-error-2}
Suppose that $\phi \in (\frac{\pi}{2}, \pi)$ (obtuse case). Let $\sigma \in (\frac{1}{2}, \frac{2}{3})$. Then
\[
    S_\lambda^{\dprime} \toL{2} 0, \text{ as } \lambda \to \infty 
    \]
\end{lemma}
\begin{proof}
We want to show that, as $\lambda \to \infty$,
\begin{align}
    \Exp[( S_\lambda^{\dprime})^2] = \Exp[(S_\lambda^{\dprime}\ind_{\{M_{\lambda} \leq  3\lambda^{1-\sigma}\}})^2] + \Exp[(S_\lambda^{\dprime}\ind_{\{M_{\lambda} > 3\lambda^{1-\sigma}\}})^2]\ \to 0 \label{eqt:eqiii}.
\end{align}
We start by bounding the second term in the sum in~\eqref{eqt:eqiii}.  Observe that   $|S_\lambda^{\dprime}|$ is bounded by $M_{\lambda} \sqrt{2}$, 
since $|\Delta_{i}| \leq \sqrt{2}$. Then by the Cauchy--Schwarz inequality, a bound on the $4$th moments of the Poisson distribution (i.e., if $X \sim \Poi(\mu)$ then $\Exp[X^{4}] 
= \mu^4+6 \mu^3+7 \mu^2+ \mu
\leq 15\mu^{4}$, for $\mu >1$), and the inequality~\eqref{eq:poisson-3-mu}, we get
\begin{align}
    \Exp[(S_\lambda^{\dprime}\ind_{\{M_{\lambda} > 3\lambda^{1-\sigma}\}})^2] 
        \leq 2(\Exp[M_{\lambda}^{4}])^\frac{1}{2}\Pr(M_{\lambda} > 3\lambda^{1-\sigma})^\frac{1}{2} 
        \to 0 ,
\text{ as } \lambda \to \infty \label{myeqt0}.
\end{align}

For the first term in the sum in~\eqref{eqt:eqiii}, write $\tDelta_i = \Delta_i \ind_{A_i }$ (set $\tDelta_i =0$ for $i > M_\lambda$), and note  
\begin{align}
    \Exp[(S_\lambda^{\dprime}\ind_{\{M_{\lambda} \leq 3\lambda^{1-\sigma}\}})^2] \leq \sum_{i,j=1}^{\floor*{3\lambda^{1-\sigma}}} \Exp[|\tDelta_i||\tDelta_j|] 
    = 
    \sum_{i,j=1}^{\floor*{3\lambda^{1-\sigma}}} \Exp[|\tDelta_i||\tDelta_j|\ind_{A_i \cap A_j}] 
    \label{eqt: eqti-1}.
\end{align}
We repeatedly apply the Cauchy--Schwarz inequality to get
\begin{align}
    \Exp[|\tDelta_i||\tDelta_j|\ind_{A_i \cap A_j}] \leq 
\Exp[(\tDelta_{i} \tDelta_j)^{2}]^{\frac{1}{2}}\Pr(A_i \cap A_j)^{\frac{1}{2}}
    \leq 
    \Exp[\tDelta_i^{4}]^{\frac{1}{4}}\Exp[\tDelta_j^{4}]^{\frac{1}{4}}\Pr(A_i )^{\frac{1}{4}} \Pr ( A_j)^{\frac{1}{4}},\nonumber
\end{align}
which yields
\begin{align}
    \Exp[(S_\lambda^{\dprime}\ind_{\{M_{\lambda} \leq 3\lambda^{1-\sigma}\}})^2] \leq \sum_{i,j=1}^{\floor*{3\lambda^{1-\sigma}}}\Exp[\tDelta_i^{4}]^{\frac{1}{4}}\Exp[\tDelta_j^{4}]^{\frac{1}{4}}
    \Pr(A_i )^{\frac{1}{4}} \Pr ( A_j)^{\frac{1}{4}}
. \label{eqt:eqti-2} 
\end{align}
We have the bound
$\Pr(A_i ) \leq \Pr(i \leq M_\lambda, \, W_i > 0)$, where $W_i$ is the number of points of $\cP_\lambda$ in $\Eo_i$. 
Given $i \leq M_\lambda$ and $Z_i = (X_i, Y_i)$,
the points of $\cP_\lambda \cap ([0,1] \times [0,Y_i))$
constitute a binomial point process of $i-1$ independent random points,
whose horizontal coordinates are i.i.d.~$U[0,1]$ random variables.
Each of these points is in $\Eo_i$ only if
its horizontal coordinate is in the interval $[X_i, X_i + c_\phi \lambda^{-\sigma}]$,
where $c_\phi = \tan ( \phi - \frac{\pi}{2} )$.
So
\[ \Pr(A_i ) \leq 1 - (1 - c_\phi \lambda^{-\sigma} )^{i-1} \leq C i \lambda^{-\sigma} \]
for some $C < \infty$ and all $i \in \N$ with $i \leq 3 \lambda^{1-\sigma} \ll \lambda^\sigma$ (recall $\sigma > 1/2$). Hence 
\begin{align}
   \sup_{1 \leq i \leq \lfloor 3\lambda^{1-\sigma} \rfloor} \Pr (A_i ) \leq C \lambda^{1-2\sigma}
   \label{eqt:eqti-6},\ \text{for}\ \sigma > 1/2.
\end{align}

We also need a bound on $|\tDelta_i|$. 
Given $A_i$ occurs,  $Z_{K(i)} \in \Eo_i$, 
and so, similarly to~\eqref{eqt:eq},
\[
    -|X_{J(i)}-X_i| \leq \Delta_{i} \leq - |X_{J(i)}-X_i| + \sec\left(\phi - \frac{\pi}{2}\right) \lambda^{-\sigma} , \text{ on } A_i.
\]
Hence, in terms of $|\tDelta_i| = |\Delta_i| \ind_{A_i}$, we have
\begin{align}
    |\tDelta_i|  \leq |X_{J(i)}-X_i|\ind_{\{   i \leq M_\lambda, \, J(i) < \infty \}} + \sec \left(\phi - \frac{\pi}{2}\right) \lambda^{-\sigma}\nonumber.
\end{align}
Using the fact that  $(a+b)^4 \leq (2\max\{a,b\})^4 = 2^4\max\{a^4,b^4\} \leq 2^4(a^4+b^4)$, we get
\begin{align}
 | \tDelta_{i}| ^{4} 
 & \leq 2^4 \left(|X_{J(i)}-X_i|^{4}\ind_{\{ i \leq M_\lambda, \, J(i) < \infty   \}} + \lambda^{-4\sigma}  \sec^{4}\left(\phi - \frac{\pi}{2}\right)\right) \label{eqt:eqti-5}.
\end{align}

Next we aim to bound the expectation of $|X_{J(i)}-X_i|^{4}\ind_{\{i \leq M_\lambda,\, J(i) < \infty \}}$. 
To this end, it is convenient to extend the sequence $X_1, \ldots, X_{M_\lambda}$ to an infinite sequence $X_1, X_2, \ldots$ of i.i.d.~$\unif{0}{1}$ random variables. For fixed $i \in \N$, denote the order statistics of the random variables $X_1, \dots, X_i$ by $X_{(1)} < X_{(2)} < \dots < X_{(i)}$ and set $X_{(0)} := 0, X_{(i+1)} :=1$. 
The 
successive \emph{spacings} associated with the order statistics are 
\begin{equation}
    \label{eq:spacings-I-def}
I_{i,j} := X_{(j)} - X_{(j-1)}, \text{ for } j \in \{1,\ldots, i+1\}.
\end{equation}
If $i \leq M_\lambda$ and $J(i)<\infty$, then $X_i-X_{J(i)}$ is equal to one of $I_{i,2}, \ldots, I_{i,i}$, so $| X_{J(i)} -X_i | \leq \max_{1 \leq j \leq i} I_{i,j}$. Hence from~\eqref{eqt:eqti-5} we get
\begin{align}
| \tDelta_{i}| ^{4} 
 & \leq 2^4 \left(\max_{1 \leq j \leq i} I^4_{i,j}+ \lambda^{-4\sigma}  \sec^{4}\left(\phi - \frac{\pi}{2}\right)\right)  \label{eqt:eqti-7}.
\end{align}
As a spacing of i.i.d.~$\unif{0}{1}$ variables, $I_{i,j}$
has $\Pr(I_{i,j}>r) =  (1-r)^{i} $, $r \in (0,1)$, for all $j \in \{ 1, \dots,i+1 \}$. 
A union bound yields,  for $r \in (0,1)$,
 \begin{align}
 \label{eq:max-I-bound}
     \Pr \Bigl( \max_{1 \leq j \leq i}I_{i,j} > r \Bigr) = \Pr(\cup_{j=1}^{i}\{I_{i,j} > r\}) \leq i(1-r)^{i} \leq i \re^{-ir},
 \end{align}
and the final bound in fact holds for all $r > 0$. 
 Using the fact that, for any random variable $\xi \in [0,1]$ and any $r>0$, it holds that $\xi^4 \leq r^4 + \1{ \xi > r }$, 
 and choosing $r = \frac{5 \log (i+1)}{i}$,
 we obtain
 \begin{align}\label{eq:max-I-4th-moment}
     \Exp \Bigl[ \max_{1 \leq j \leq i}I_{i,j}^{4} \Bigr] &\leq \left( \frac{5 \log (i+1)}{i}\right)^{4}
     +\Pr\left(\max_{1 \leq j \leq i}I_{i,j} > \frac{5 \log (i+1)}{i} \right) \nonumber \\
    & \leq 
    \left( \frac{5 \log (i+1)}{i}\right)^{4} +\frac{1}{i^{4}} = O(  i^{-4} \log^{4} i ), 
 \end{align}
 by~\eqref{eq:max-I-bound}.
Combining inequalities~\eqref{eqt:eqti-2},~\eqref{eqt:eqti-6},~\eqref{eqt:eqti-7} and~\eqref{eq:max-I-4th-moment} yields, for some $C < \infty$,
\begin{align}
    \Exp[(S_\lambda^{\dprime}\ind_{\{M_{\lambda} 
     \leq 3\lambda^{1-\sigma}\}})^2] 
    &\leq C \lambda^{\frac{1}{2}- \sigma}\Biggl(\sum_{i=1}^{\floor*{3\lambda^{1-\sigma}}}\left( \frac{\log (i+1)}{i} +   \lambda^{-\sigma}  \right)\Biggr)^{2} \nonumber\\
    & \leq  C \lambda^{\frac{1}{2}- \sigma} \left(\log^2 (3\lambda^{1-\sigma})+3\lambda^{1-2\sigma}\right)^{2}  \to 0, 
     \label{myeq2} 
\end{align}
since $\sigma > 1/2$. 
Finally, we combine~\eqref{myeqt0} and~\eqref{myeq2} to yield~\eqref{eqt:eqiii} and complete the proof.
\end{proof}

\begin{proof}[Proof of Theorem~\ref{thm:DLF-singly-aligned-approx}: Obtuse case]
Recall that $\cLbndry - \cLdlf{\cV_\lambda} =   S_\lambda^{\prime} + S_\lambda^{\dprime}$ as defined at~\eqref{eq:S-primes-def}. 
Combining Lemma~\ref{lem:obtuse-boundary-error-1} and Lemma~\ref{lem:obtuse-boundary-error-2} and applying Minkowski’s inequality shows that $S_\lambda^{\prime} + S_\lambda^{\dprime} \toL{2} 0$, as $\lambda \to \infty$, 
provided $\sigma \in (\frac{1}{2}, \frac{2}{3})$. This completes the proof.
\end{proof}

\subsection{Coupling for the acute case}
\label{sec:coupling-acute}

In this section we assume that $\phi \in (0,\pi/2)$.  We follow a similar approach to that employed in Section \ref{sec:coupling-obtuse}, modifying the argument where necessary for the acute cone geometry.  In particular, we begin by defining the region $\Ea_i$ of $\cRbndry$ for each point of $\cW_\lambda$, which plays an important role in controlling the size of $\Delta_{i}$ in this case.  Here, any Poisson point in the region $\Ea_i$ (associated to point $Z_i$) is not a candidate neighbour in the (2-dimensional) MDSF, but does generate a candidate neighbour in the associated DLF.
 
  \begin{definition}
  For each $i = 1,2, \dots, M_{\lambda}$, define
  \begin{align}
      \Ea_i := \{(x,y) \in \cRbndry : x< X_i, y<Y_i, (x,y) \notin \cCtp(Z_i)
      \}. \nonumber
  \end{align}
  \end{definition}
  Note that $\Ea_i$ is always contained in a right-angled triangle with height $\lambda^{-\sigma}$ and width $ \lambda^{-\sigma}\cot\phi$.
Recall the indices $J(i)$ and $K(i)$ defined for each $i\in \{1,\dots, M_\lambda\}$, given in Definitions~\ref{def:DLF-J} and~\ref{def:DLF-K}.  For the acute case, we introduce two further indices, as follows.
  
  \begin{definition}
\label{Def:acuteofX}
For $i \in \{1, \dots ,M_\lambda\}$, let $\tJ (i) := \max \{j <i : X_j = \max_{k<i}\{X_k: X_k < X_i - \lambda^{-\sigma}\cot\phi \} \}$ denote the largest index of a rightmost predecessor of $i$, arriving to the left of $X_i - \lambda^{-\sigma}\cot\phi$ (with the convention that $\tJ(i) := \infty$ if there is none).
\end{definition}

\begin{definition}\label{defacutep}
For $i \in \{1,\dots, M_\lambda\}$, let $I(i) \leq i$ be defined by $I(i) := \min\{ j \leq i : X_j = \min_{k \leq i}\{X_k : X_{i}-\lambda^{-\sigma}\cot \phi \leq X_{k} \leq X_{i} \} \}$.
    \end{definition}
    
Note that, if the interval $[X_i - \lambda^{-\sigma}\cot \phi, X_i)$ contains none of $X_1, \dots, X_{i-1}$, we have both $I(i) = i$ and $\tJ(i) = J(i)$; otherwise, $I(i) < i$ is the smallest index of a leftmost predecessor of $i$, arriving in $[X_i - \lambda^{-\sigma}\cot \phi, X_i)$, and moreover $\tJ(i) \neq J(i)$.  In either case, when $\tJ(i) < \infty$ the points $X_{\tJ(i)}$ and $X_{I(i)}$ are a consecutive pair of order statistics of the random variables $X_{1}, X_{2}, \dots , X_{i}$.
The definitions of $\Ea_i$ and $\tJ(i)$ also imply that if 
$\tJ(i) < \infty$ then $Z_{\tJ(i)} \notin \Ea_i$. 

For $i \in \N$ we define
\begin{equation}
\label{eq:A-i-def-acute}
A_{i} :=\{ i \leq M_\lambda, \, J(i) < \infty, \, K(i) < \infty, \, Z_{J(i)} \in \Ea_i\}.
\end{equation}
Notice here that the definition of $A_{i}$ at~\eqref{eq:A-i-def-acute} refers to the event $Z_{J(i)} \in \Ea_i$, in contrast to our use of $A_i$ in the previous section, defined by~\eqref{eq:A-i-def}, which refers to event $Z_{K(i)} \in \Eo_i$.
With this (re)definition of $A_i$, we recall $S_\lambda^{\prime}$ and $S_\lambda^{\dprime}$ as defined at~\eqref{eq:S-primes-def}.
The next two lemmas show that both $S_\lambda^{\prime}$
and $S_\lambda^{\dprime}$ tend to~$0$ in $L^2$, and are analogous to Lemmas~\ref{lem:obtuse-boundary-error-1} and~\ref{lem:obtuse-boundary-error-2} above.

  \begin{lemma}
  \label{lem:acute-boundary-error-1}
  Suppose that $\phi \in (0,\frac{\pi}{2})$ (acute case).
Let $\sigma \in (\frac{1}{2},\frac{2}{3})$. Then 
\[
    S_\lambda^{\prime} \toL{2} 0,  \text{ as } \lambda \to \infty.  \]
\end{lemma}
\begin{proof}
Given that $A_{i}^{\rc}$ occurs, since $i \leq M_\lambda$ exactly one of the following events 
occurs:
\begin{enumerate}[label={\upshape(\roman*)}, align=left, widest=iii, leftmargin=*]
        \item \label{lem:acute-boundary-error-1-i} \{$J(i) = \infty\}$,
     \item \label{lem:acute-boundary-error-1-ii} \{$J(i) < \infty, \, K(i) = \infty\}$, 
       \item \label{lem:acute-boundary-error-1-iii}
       \{$J(i) < \infty,\, K(i) < \infty,\, Z_{J(i)} \notin \Ea_i\}$.
\end{enumerate}
Note that cases~\ref{lem:acute-boundary-error-1-i},~\ref{lem:acute-boundary-error-1-ii}, and~\ref{lem:acute-boundary-error-1-iii} correspond to the cases~\ref{lem:obtuse-boundary-error-1-i},~\ref{lem:obtuse-boundary-error-1-ii}, and~\ref{lem:obtuse-boundary-error-1-iii} appearing in proof of Lemma~\ref{lem:obtuse-boundary-error-1} but with the roles of $J(i)\ \& \ K(i) $ swapped. 

Indeed, in Case~\ref{lem:acute-boundary-error-1-i}, $J(i) = \infty$ implies that $K(i) = \infty$ also; hence $\Delta_{i} = 0$ as with Case~\ref{lem:obtuse-boundary-error-1-i} in Lemma~\ref{lem:obtuse-boundary-error-1}. 
In Case~\ref{lem:acute-boundary-error-1-ii}, $K(i) = \infty$ and $J(i) < \infty$, so that $\cCotp(Z_i)$ contains no point in $\cP_\lambda$, but in the DLF there is a candidate nearest neighbour ${J(i)}<i$ of $i$ with $X_{J(i)} < X_{i}$. Hence $Z_{J(i)} \in \Ea_i$, and 
\begin{equation}
    0  \leq |\Delta_i | = |X_{J(i)}-X_i| \leq \lambda^{-\sigma}\cot \phi\label{eqt:acu202}.
\end{equation}
In Case~\ref{lem:acute-boundary-error-1-iii}, both $J(i) < \infty$ and $K(i) < \infty$, and $Z_{J(i)} \notin \Ea_i$. Hence the same argument as in Case~\ref{lem:acute-boundary-error-1-iii} of Lemma~\ref{lem:obtuse-boundary-error-1} shows that
$0 \leq \Delta_{i} \leq \lambda^{-\sigma}$.
Combining cases~\ref{lem:acute-boundary-error-1-i},~\ref{lem:acute-boundary-error-1-ii}, and~\ref{lem:acute-boundary-error-1-iii} we see that $0 \leq |\Delta_i | \ind_{A^{\rc}_{i}} \leq \lambda^{-\sigma} \max (1,\cot \phi)$, and the proof is concluded similarly to Lemma~\ref{lem:obtuse-boundary-error-1}.
\end{proof}

We show in the following lemma that $S_\lambda^{\dprime}$ converges in $L^2$ to $0$ as the intensity $\lambda$ tends to infinity. Let $A_{i}=\{J(i) < \infty, K(i) < \infty, Z_{J(i)} \in \Ea_i\}$. We have by definition of $S_\lambda^{\dprime}$ that $  S_\lambda^{\dprime} =  \sum_{i=1}^{M_{\lambda}}\Delta_{i}\ind_{A_{i}}$, where  $\Delta_{i} = \|Z_i-Z_{K(i)}\|\ind_{\{K(i) < \infty\}}  - |X_i-X_{J(i)}|\ind_{\{J(i) < \infty\}}$, and $M_{\lambda} \sim \Poi(\lambda^{1-\sigma})$.

\begin{lemma}\label{lem:acute-boundary-error-2}
  Suppose that $\phi \in (0,\frac{\pi}{2})$ (acute case).
Suppose that  $\sigma \in (\frac{1}{2}, \frac{2}{3})$. Then,
\begin{equation}
    S_\lambda^{\dprime} \toL{2} 0,  \text{ as } \lambda \to \infty. \nonumber
\end{equation}
\end{lemma}
\begin{proof}
We follow the proof of Lemma~\ref{lem:obtuse-boundary-error-2} where possible but making appropriate changes specific to the acute case where necessary.
We use the partition in~\eqref{eqt:eqiii} once more, and~\eqref{myeqt0} holds as written, showing that 
$\Exp[(S_\lambda^{\dprime}\ind_{\{M_{\lambda} >  3\lambda^{1-\sigma}\}})^2] \to 0$ as $\lambda \to \infty$.
For the other term in~\eqref{eqt:eqiii}, the bound~\eqref{eqt:eqti-2} carries over. Moreover, 
a very similar argument (with a different definition of the quantity $c_\phi$) yields the bound on $\Pr ( A_i)$ (with its new definition) from~\eqref{eqt:eqti-6}.

The main divergence in the proof is in the upper bound on $| \tDelta_i |$, where $\tDelta_i = \Delta_i \2 {A_i}$ is defined in the same way as before. 
For the obtuse case we showed that  $|\Delta_{i}|$ was close to the one-dimensional distance $|X_{J(i)} - X_{i}|$ (see inequality~\eqref{eqt:eqti-5}) and hence could be bounded in terms of the spacings $I_{i,j}$ defined at~\eqref{eq:spacings-I-def}. In the acute case, as we will see, the corresponding argument only yields an upper bound on $\Delta_{i}$ in terms of two-dimensional distance $\|Z_{K(i)}-Z_{i}\|$ and we need an additional argument to bound $\Delta_{i}$ in terms of an appropriate  one-dimensional distance.
Given $J(i) < \infty$ and $K(i) < \infty$, by the definition of $J(i)$, we have $X_{K(i)} \leq X_{J(i)}$, so
\begin{equation}
    |X_{J(i)}-X_{i}| \leq |X_{K(i)}-X_{i}| \leq \|Z_{K(i)}-Z_{i}\|.
\end{equation}
Hence $0 \leq \Delta_{i} \leq \|Z_{K(i)}-Z_{i}\|$.
Note that $\|Z_{K(i)}-Z_{i}\|$ is the long two-dimensional distance. So we need to work out to find the upper bound in terms of one-dimensional distance, and this is provided by the following two statements.

First, if $\tJ(i) = \infty$, then set $\{X_j: X_j < X_i-\lambda^{-\sigma}\cot(\phi), j<i \}$ is empty. But $K(i) < i$ since $ Z_{K(i)} \in \cCtp(Z_i)$ so the inequality $X_{K(i)} \geq X_{i} - \lambda^{-\sigma} \cot \phi$ must hold to avoid a contradiction. Then, by triangle inequality
\begin{align} \label{eqacutrp}
    \|Z_{K(i)}-Z_{i}\| \leq \lambda^{-\sigma} +|X_{K(i)}-X_{i}| \leq  \lambda^{-\sigma} +  \lambda^{-\sigma} \cot(\phi) \leq \lambda^{-\sigma}(1+\cot(\phi)).
\end{align}

Second, if $\tJ(i) < \infty$, we know $Z_{\tJ(i)} \notin \Ea_i$ and so $Z_{\tJ(i)} \in \cCtp(Z_i)$. Then by Definition~\ref{defacutep} of $I(i)$ along with triangle inequality, we have
\begin{align}
    \|Z_{K(i)}-Z_{i}\|\ind_{\{\tJ(i) < \infty\}} & \leq \|Z_{\tJ(i)}-Z_{i}\|\ind_{\{\tJ(i) < \infty\}} \nonumber \\
& \leq 
\left(|X_{\tJ(i)}-X_{i}|+\lambda^{-\sigma}\right) \ind_{\{\tJ(i) < \infty\}} \nonumber \\
& \leq 
\left(
     |X_{\tJ(i)}-X_{I(i)}|+ |X_{I(i)}-X_{i}|+ \lambda^{-\sigma}
     \right)\ind_{\{\tJ(i) < \infty\}} \nonumber \\
   & \leq \left(|X_{\tJ(i)}-X_{I(i)}| + \lambda^{-\sigma} (1+ \cot \phi )\right)\ind_{\{\tJ(i) < \infty\}}. \nonumber
\end{align}
In terms of $\tDelta_{i} = \Delta_{i}\ind_{A_{i}} $, the upper bound of $\tDelta_{i}$ is given by the following inequality
\begin{align} \label{deltaacute}
    0 \leq \tDelta_{i} \leq |X_{\tJ(i)}-X_{I(i)}|\ind_{\{\tJ(i) < \infty\}} + \lambda^{-\sigma} (1+ \cot(\phi)) \leq \max_{1 \leq j \leq i}I_{i,j}+\lambda^{-\sigma} (1+ \cot(\phi)).
\end{align}
For the final inequality, we used that, for $i \leq M_\lambda$, the variables $X_{\tJ(i)}$, $X_{I(i)}$ are a consecutive pair of order statistics and therefore their absolute difference is equal to $I_{i,j}$ for some $j \in \{2, \dots, i\}$.  Taking 4th powers,
 \begin{equation}\label{4thmmde}
   |  \tDelta_{i}^{4} | \leq 2^{4} \left( \max_{1 \leq j \leq i}I_{i,j}^{4}+ \lambda^{-4\sigma} (1+ \cot \phi)^{4} \right). 
 \end{equation}
Now the proof follows that of Lemma~\ref{lem:obtuse-boundary-error-2} but with~\eqref{4thmmde} in place of~\eqref{eqt:eqti-7}.
 \end{proof}

The proof of Theorem~\ref{thm:DLF-singly-aligned-approx} in the acute case now follows exactly as for the obtuse case.

\begin{proof}[Proof of Theorem~\ref{thm:DLF-singly-aligned-approx}: Acute case]
Recall from~\eqref{eq:S-primes-def} that $\cLbndry - \cLdlf{\cV_\lambda} =   S_\lambda^{\prime} + S_\lambda^{\dprime}$. 
Combining Lemma~\ref{lem:acute-boundary-error-1} and Lemma~\ref{lem:acute-boundary-error-2} and applying Minkowski’s inequality completes the proof.
\end{proof}

\subsection{Convergence to the boundary limit}
\label{sec:boundary-effects-proof}

In this section, we derive the proof of Theorem~\ref{thm:boundary-effects} for both singly-aligned cones since the proof is identical for both obtuse case and acute case. Theorem~\ref{thm:boundary-effects} plays a significant role in our analysis in which the approximation process to one-dimensional distance holds due to the boundary effect. These boundary effects can be characterized by the fixed point equation given by~\eqref{eq:fixed-point}. Recall that Theorem~\ref{thm:boundary-effects} states that the total edge length ${\tcLbndry}$ converges in distribution to $Q$, where the distribution of $Q \sim \cQ$ given by the fixed point equation~\eqref{eq:fixed-point} as the number of points tends to infinity. 

\begin{proof}[Proof of Theorem~\ref{thm:boundary-effects}]
Recall that, $\cLdlf{\cU_n}$ is the total length of the DLF defined on $n$ independent $\unif{0}{1}$ random points. Define $q_n:= \Exp \cLdlf{\cU_n}$ (and $q_0:=0$).
It is shown in (3.17) of~\cite{pw-south-west} that,
with $\gamma$ denoting Euler’s constant~\cite[p.~255]{as64},
\begin{equation}
    q_n = \log (1+n) +c + \delta_n, \text{ where }  c := \gamma - 2 , 
\label{eqt:eqt1.0}
\end{equation}
and $\delta_n$ is such that $\sup_{n \in \ZP} | \delta_n | < \infty$ and
$\lim_{n\to\infty} \delta_n =0$.

Conditional on $M_\lambda$ (the size of $\cV_\lambda$), the
sequence $\cV_\lambda$ has the same distribution as $\cU_{M_\lambda}$. In particular,
$\Exp [ \cLdlf{\cV_\lambda} \mid M_\lambda ] = q_{M_\lambda} = \sum_{n=0}^\infty q_n \ind_{\{M_\lambda = n\}}$, and $\Exp \cLdlf{\cV_\lambda} = \Exp  q_{M_\lambda} $.
Write 
\begin{align}
    {\tcLbndry} &= (\cLbndry -\cLdlf{\cV_\lambda} ) + (\cLdlf{\cV_\lambda} - q_{M_{\lambda}}) + (q_{M_{\lambda}} - \Exp q_{M_{\lambda}} ) +(\Exp q_{M_{\lambda}}  - \Exp \cLbndry ). \label{eq:finalstage}
\end{align}
The first term on the right-hand side of~\eqref{eq:finalstage} tends to $0$ in $L^2$ by Theorem~\ref{thm:DLF-singly-aligned-approx}, and the final term also tends to $0$ since the $L^2$ convergence in Theorem~\ref{thm:DLF-singly-aligned-approx} implies that $ \Exp  \cLbndry - \Exp q_{M_\lambda}  \to 0$.
For the second term, since $M_{\lambda} \to \infty$ in probability, as $\lambda \to \infty$,
it follows from Theorem~\ref{thm:PW-DLF}
that 
$\cLdlf{\cV_\lambda} - q_{M_{\lambda}}$
converges in distribution to $Q$, as $\lambda \to \infty$, where $Q \sim \cQ$ given by equation~\eqref{eq:fixed-point}. It remains to consider the third term on the right-hand side of~\eqref{eq:finalstage}; we claim that $q_{M_{\lambda}} - \Exp q_{M_{\lambda}} \to 0$ in probability.

From~\eqref{eqt:eqt1.0}, we have
\begin{align*}
    q_{M_{\lambda}} - \Exp q_{M_{\lambda}}  & = \log (1+M_\lambda) -\Exp \log (1+M_\lambda) + \delta_{M_\lambda} - \Exp \delta_{M_\lambda}.
\end{align*}
Here, since $M_\lambda \to \infty$, in probability, as $\lambda \to \infty$ we have
$\delta_{M_\lambda} \to 0$, in probability, and (by bounded convergence)
$\Exp \delta_{M_\lambda} \to 0$ as $\lambda \to \infty$.
It is not hard to show, using Poisson concentration, that $ \log (1+M_\lambda) -\Exp \log (1+M_\lambda) \to 0$, in probability, as $\lambda \to \infty$. Hence on the right-hand side of~\eqref{eq:finalstage}, the first, third, and fourth terms all tend to~$0$ in probability (say), while the second converges in distribution to $Q$. 
Thus the result follows from Slutsky's theorem.
\end{proof}

\section{Asymptotic independence}
\label{sec:asymptotic-independence}

\subsection{Intermediate region for singly-aligned cones} 
\label{sec:intermediate-region}

Recall
by  Definition~\ref{def:3-regions} for $\sigma \in (\frac{1}{2},\frac{2}{3})$ that $\cRint = [0,1] \times [\lambda^{-\sigma}, RS]$, for $R$ given in Lemmas~\ref{obtuse'lemma} and~\ref{acute'lemma} for singly-aligned cones; also recall  $\cRbndry = [0,1 ] \times [0,\lambda^{-\sigma}]$.  
Recall by Definition~\ref{allregions} that $\cLint = \sum_{\bx \in \cP_\lambda \cap \cRint} \cDtp(\bx,\cP_\lambda)$, where $ \cDtp(\bx, \cP_\lambda)$ is the distance from point $\bx$ to its nearest neighbour in $\cP_\lambda \cap \cCtp(\bx)$, and ${\tcLint} = \cLint - \Exp[\cLint]$. %is centered random variable.
The following result shows that edges from the intermediate region are negligible in the limit.
The approach of this section is similar to that from~\cite[\S 7]{pw-south-west}.

\begin{theorem} \label{thm:proofa}
Suppose that $(\theta,\phi)$ is singly-aligned. 
Then, $\text{as}\ \lambda \to \infty$,
\begin{equation}
    \tcLint \toP 0 \label{eqn:a}.
\end{equation}
\end{theorem}

We partition $\cRint$ region into $\ell_{\lambda}:=\floor*{ \frac{\lambda^{1-\sigma}}{\log^{2} \lambda}}$ cells $\Gamma_i, i= 1, \dots, \ell_{\lambda}$ (from left to right) of height $RS-\lambda^{-\sigma}$ and width ${1}/{\ell_{\lambda}}$. 
Similarly, partition $\cRbndry$ region into cells $\beta_i$, $i = 1, \dots, \ell_{\lambda}$ (also from left to right) of height $\lambda^{-\sigma}$ and width $1/{\ell_{\lambda}}$.
For each $i = 1, \dots,\ell_{\lambda}$ define $X_i := \sum_{\bx \in \cP_\lambda \cap \Gamma_i} \cDtp(\bx,\cP_\lambda)$ and define $\tau_i := \{|\cP_\lambda \cap \beta_i| >0 \}$.

\begin{lemma} \label{lem2022}
Let $i,j \in \{ 1, \dots, \ell_{\lambda} \}$ with  $|i-j| > 4$. 
For all $\lambda$ large enough, the 
random variables $\ind_{\tau_{i-2}}X_i$ and  $\ind_{\tau_{j-2}}X_j$ are independent (with the convention that $\tau_{-1} = \tau_{0} = \Omega$).
\end{lemma}

\begin{proof}
Define $\bar{A}_i := \cup_{\max(1,i-3)\leq \ell \leq \min(\ell_\lambda,i+1)}(\Gamma_{\ell} \cup \beta_{\ell})$.
First we show that given $\tau_{i-2}$, the nearest-neighbour of any point in $\cP_\lambda \cap \Gamma_i$ must lie in $\bar{A}_i$ (for sufficiently large $\lambda$). Let $z_1 \in \cP_{\lambda} \cap  \Gamma_i$. For all $i= 1, \dots, \ell_{\lambda}$, the nearest neighbour of $z_{1}$ can not be in  $\cup_{j> i+1}(\Gamma_{j} \cup \beta_{j})$. 
This is immediate for the acute cone, for the obtuse cone we need to check that for $z_2 \in \cCtp(z_1) $ with $x_{2} > x_{1}$ then $z_{2} \in \bar{A}_{i}$. Hence,
\begin{align}\label{disjiont}
    x_{2} - x_{1} < RS \tan\left(\phi - \frac{\pi}{2}\right) < \frac{1}{\ell_{\lambda}},
\end{align}
which holds for large enough $\lambda$, since $S \leq \sqrt{2c\log \lambda/ \lambda}$ for sufficiently large $\lambda$.

Hence for $i=1,2,3,4$, the nearest-neighbour of $z_{1}$ can only be found in $\bar{A}_i$. For $i\geq 5$, given that $\tau_{i-2}$ occurs, take $z_{3} \in \cP_{\lambda} \cap \beta_{i-2}$.
Note that $z_{3}$ is in $\cCtp(z_1)$ for large enough $\lambda$, and hence a candidate nearest neighbour for $z_1$. This is immediate for the obtuse case (since $x_{3} < x_{1}$ and $y_{3}< y_{1}$), and for the acute case, we observe that 
    $x_{3} < x_{1} - 1/{\ell_{\lambda}} < x_{1} - RS \cot\phi$, 
similarly to \eqref{disjiont}, which together with $y_{3} < y_{1}$ implies $z_3 \in \cCtp(z_1)$.

Now suppose $z_2 \in \cCtp(z_1)$ but not in $\bar{A}_i$.
Therefore $z_2 \in \cup_{j<i-3} (\Gamma_j \cup \beta_j)$, and we must have $x_2 < x_3 - \frac{1}{\ell_{\lambda}}$.
We show that the Euclidean distance between points $z_1$ and $z_3 \in \beta_{i-2}$ is smaller than the distance between $z_1 \in \Gamma_i$ and $z_2 \notin \bar{A}_i$. Since $S=o(1/{\ell_{\lambda}})$, for large enough $\lambda$, 
\begin{align}
    \|z_2-z_1\| \geq |x_2-x_1|> |x_3-x_1| + \frac{1}{\ell_{\lambda}} >  |x_3-x_1|+ |y_3-y_1|\geq \|z_3 -z_1\| \nonumber,
\end{align}
hence $z_2$ is not a nearest neighbour of $z_1$.

We deduce that the random variable $\ind_{\tau_{i-2}}X_i$ is determined by $\cP_\lambda \cap \bar{A}_i$. Indeed, either $\tau_{i-2}$ does not occur (not possible for $i = 1$ or $2$) meaning 
$\beta_{i-2}$ has no Poisson point and
$\ind_{\tau_{i-2}}X_i \equiv 0$; otherwise $\tau_{i-2}$ occurs and  $X_i = \sum_{\bx \in \cP_\lambda \cap \Gamma_i} \cDtp(\bx,\cP_\lambda)$  is the sum of nearest neighbour distances, where each distance $\cDtp(\bx,\cP_\lambda)$ is determined by  $\cP_\lambda \cap \bar{A}_i$. 

Finally, for $|i - j| > 4$, the regions $\bar{A}_i$ and $\bar{A}_j$ are disjoint, so the random variables $\ind_{\tau_{i-2}}X_i$ and $\ind_{\tau_{j-2}}X_j$ depend on disjoint regions of Poisson process $\cP_{\lambda}$, so are independent.
\end{proof}

\begin{proof}[Proof of Theorem~\ref{thm:proofa}.]
It is enough to show that, as $\lambda \to \infty$,
$\Var (\cLint) \to 0$.
Recalling that
$X_i = \sum_{\bx \in \cP_\lambda \cap \Gamma_i} \cDtp(\bx,\cP_\lambda)$,
$ i =1, \dots, \ell_{\lambda}$, we can write 
\begin{align}
   \Var(\cLint) = \Var\left(\sum_{i=1}^{\ell_{\lambda}}X_i\right) = \Var\left( 
   \sum_{i=1}^{\ell_{\lambda}}X_{i}\ind_{\tau_{i-2}}+ \sum_{i=1}^{\ell_{\lambda}}X_{i}\ind_{\tau_{i-2}^{\rc}}
   \right) \label{eqR}.
\end{align}

Consider $\sum_{i=1}^{\ell_\lambda} X_i'$, where we set $X_i^{\prime} := X_{i}\ind_{\tau_{i-2}}$. Observe that
   $\Cov(X_i^{\prime},X_j^{\prime}) = 0 $, if $|i-j|>4$ for large enough $\lambda$ by Lemma~\ref{lem2022}. Therefore, 
   \begin{align}
\Var\left(\sum_{i=1}^{\ell_{\lambda}}X_i^{\prime}\right)
    =\sum_{\substack{i,j=1 \\ 
    |i-j| \leq 4}}^{\ell_{\lambda}} \Cov(X_i^{\prime},X_j^{\prime}) 
   & \leq 9\ell_{\lambda} \max_{1 \leq i,j \leq \ell_\lambda} \Cov(X_i^{\prime},X_j^{\prime}) 
 \leq 9\ell_{\lambda} \max_{1 \leq i \leq \ell_\lambda }\Var(X_i^{\prime})\label{eQ},
\end{align}
where the last inequality follows from the Cauchy--Schwarz inequality.
Now observe that 
\begin{align}
    |X_i^{\prime}| \leq |\cP_\lambda \cap \Gamma_i|\frac{5}{\ell_{\lambda}}, \label{eW}
    \end{align}
    since (see proof of Lemma~\ref{lem2022}) given $\tau_{i-2}$, the nearest neighbour of any $\cP_{\lambda} \cap \Gamma_{i}$ is in $\bar{A}_i$. Here $|\cP_\lambda \cap \Gamma_i|$ is Poisson with parameter bounded by 
    \begin{align}
        \max_{1 \leq i \leq \ell_\lambda} \Exp[|\cP_\lambda \cap \Gamma_i|] \leq \frac{\lambda RS}{\ell_{\lambda}} =O\left(\lambda^{\sigma -\frac{1}{2}}\log^{\frac{5}{2}}\lambda\right), \nonumber
    \end{align}
    since $S \leq \sqrt{2c\log \lambda/ \lambda}$ for sufficiently large $\lambda$. 
   Hence, squaring~\eqref{eW} and taking expectations,  
\begin{align}
   \max_{1 \leq i \leq \ell_\lambda} \Var(X_i^{\prime}) = O \left(\lambda^{4\sigma - 3}\log^{9}\lambda\right)\label{eY},
\end{align}
as $\lambda \to \infty$. 
Hence by~\eqref{eQ} and~\eqref{eY}, recalling $\sigma \in ( \frac{1}{2}, \frac{2}{3})$, we obtain
\begin{align}\label{eD}
    \Var\left(
   \sum_{i=1}^{\ell_{\lambda}}X'_i \right) 
   = O\left (\lambda^{3\sigma-2} \log^{7}\lambda \right)
   \to 0, \text{ as } \lambda \to \infty.
\end{align}

Next consider $\sum_{i=1}^{\ell_{\lambda}}X_i\ind_{\tau_{i-2}^{\rc}}$. Using the crude bound $X_i \leq \sqrt{2}|\cP_\lambda \cap \cRint|$, we have
\begin{align*}
    \Exp\left(\left(\sum_{i=1}^{\ell_{\lambda}}X_i\ind_{\tau_{i-2}^{\rc}}\right)^{2} \right) \leq 2 \Exp\left(|\cP_\lambda \cap \cRint|^{2} \left(\sum_{i=3}^{\ell_{\lambda}}\ind_{\tau_{i-2}^{\rc}}\right)^{2} \right) \leq 2 \Exp\left(| \cP_\lambda \cap \cRint|^{2}
    \right) \ell_{\lambda}^{2}\cdot \Pr\left(\tau_{1}^{\rc}\right), 
\end{align*}
where the final inequality follows from independence of $|\cP_\lambda \cap \cRint|$ and the i.i.d.~variables $\ind_{\tau_{i-2}^{\rc}}$, $i = 3, \dots, \ell_{\lambda}$. Then, for all $\lambda$ large enough
$\Exp(|\cP_\lambda \cap \cRint|^{2}) \leq \Exp(|\cP_\lambda|^{2}) \leq 2 \lambda^{2}$
and 
$\Pr\left(\tau_{1}^{\rc}\right) \leq \exp\{-\log^{2} \lambda\} = O(\lambda^{-10})$, say. 
Combining these bounds we obtain 
\begin{align}\label{last}
    \Var\left(\sum_{i=1}^{\ell_{\lambda}}X_i\ind_{\tau_{i-2}^{\rc}} \right) = O\left(\lambda^{2} \ell_{\lambda}^{2} \lambda^{-10}\right) = O\left(\lambda^{2+2-2\sigma-10}\right) \to 0,\ \text{as}\ \lambda \to \infty. 
\end{align}
Finally, by using~\eqref{eqR} with~\eqref{eD} and~\eqref{last} and Cauchy--Schwarz, we obtain the result.
\end{proof}

\subsection{Asymptotic independence in the limit theorem}
\label{sec:asymptotic-independence-proof}

This section, dealing only with singly-aligned cones,
establishes the independence between the two
contributions in the limit result~\eqref{eq:generalresult-singly-aligned},
arising from the asymptotic independence between the random variables ${\tcLbulk}$ and ${\tcLbndry}$ in a sense we make precise below.

Recall the scale parameter $S$ defined in \eqref{eq:S-choice}, which is the reciprocal of some positive integer, $k_\lambda$, say. 
Let $\cC:= \{p=(p_1,p_2) \in \Z^{2} : 1 \leq p_1, p_2 \leq k_{\lambda}\}$ be an indexing of squares of side length~$S$ that tile   $[0,1]^2$.  Recall, from Definition~\ref{dfin-tiling-in-R^2}, the notations $T(p)$ for $p \in \Z^2$ and $T(V)$ for $V\subseteq \Z^2$.

Set
\[
\begin{split}
\cC_{1}& := \{p=(p_1,p_2) \in \Z^{2} : 1 \leq p_1 \leq k_{\lambda}, R+1 \leq p_{2} \leq  k_{\lambda}\},\\
\cC_{2} & := \{p=(p_1,p_2) \in \Z^{2} : 1 \leq p_1 \leq k_{\lambda}, 2 \leq p_{2} \leq  R\},
\end{split}
\]
and observe that $T(\cC_1) = \cRbulk$ and $T(\cC_{2}) \subseteq  \cRint$, and for all $p \in \cC \setminus (\cC_1 \cup \cC_2)$ we have $T(p) \cap \cRbndry \neq \emptyset$, using notations from Definition~\ref{def:3-regions}.

Let $N_{\lambda,p} := | \cP_\lambda \cap T(p)|$ and define
$F_\lambda := \{N_{\lambda,p} >0 \ \forall\ p \in \cC_1 \cup \cC_2\}$.
Then $N_{\lambda,p} \sim \Poi (\mu_\lambda)$ where, for all $\lambda$ large enough, $\mu_\lambda = \lambda S^2 \geq  (c/2) \log \lambda$. Then 
\[ \Pr ( F_\lambda^\rc ) \leq \Pr \left( \bigcup_{p \in \cC_1 \cup \cC_2 } \{ N_{\lambda, p} = 0 \} \right) \leq C \lambda^2 \re^{-\mu_\lambda} \to 0 ,\]
 for $c>0$ sufficiently large. 
Recall the definitions of ${\tcLbulk}$ and ${\tcLbndry}$ from Definition~\ref{allregions}.

\begin{lemma}\label{ll}
Suppose $(\theta,\phi)$ is singly-aligned.
Then ${\tcLbulk}\ind_{F_\lambda}$ and ${\tcLbndry}$ are independent.
\end{lemma}

\begin{proof}
We show ${\tcLbulk}\ind_{F_\lambda}$ is determined by $\cP_\lambda \cap (\cRbulk\cup \cRint)$, whereas ${\tcLbndry}$ is determined by $\cP_\lambda \cap \cRbndry$. Since $\cRbulk\cup \cRint$ and $\cRbndry$ are disjoint, then the random variables ${\tcLbulk}\ind_{F_\lambda}$ and ${\tcLbndry}$ are independent by the spatial independence property of the Poisson process $\cP_\lambda$. 

Suppose ${F_\lambda}$ occurs, where ${F_\lambda}$ is determined by $\cP_\lambda \cap T(\cC_1 \cup \cC_2) \subseteq \cP_\lambda \cap (\cRbulk \cup \cRint)$. Lemma~\ref{obtuse'lemma} in the obtuse case and Lemma~\ref{acute'lemma} in the acute case imply that $\cRbulk = T(\cC_1)$ is a compatible bulk as in Definition~\ref{def:properties}, for the set of companions $C_{0}$ as given by Lemmas~\ref{obtuse'lemma} and~\ref{acute'lemma}.  As a consequence, we can show for any $\bx \in \cP_{\lambda} \cap \cRbulk$ that $N_{\theta,\phi}(\bx,\cP_{\lambda}) \notin \cRbndry$.  Indeed, let $p \in \cC_1$ be such that $\bx \in T(p)$; then either \ref{stat.1} or \ref{stat.3} from Definition~\ref{def:properties} is satisfied.

If $p$ satisfies~\ref{stat.1}, then there exists $r \in C_0$ with both $p+r \in \cC$ and $T( \cS_{p,\rho(r)}) \cap \cRbndry = \emptyset$.  Since $T(p+r) \subset T( \cS_{p,\rho(r)})$ we deduce that $p+r \in \cC_1 \cup \cC_2$, and since $F_\lambda$ occurs, there exists a point $\by \in \cP_\lambda \cap T(p+r)$.  Hence the conditions of Proposition~\ref{prop:geom} hold, implying that $N_{\theta,\phi}(\bx,\cP_{\lambda})$ exists and is not  in $\cRbndry$.
Otherwise, if $p$ satisfies~\ref{stat.3} then, since  
$\cDtp(\bx,\cP_\lambda)$ is determined by $\cP_{\lambda} \cap \cCtp(\bx)$,
it is determined entirely by points of $\cP_{\lambda} \cap (\cRbulk \cup \cRint)$.

Hence, on the event $F_\lambda$, the random variable 
$\cLbulk = \sum_{\bx \in \cP_\lambda \cap \cRbulk} \cDtp(\bx,\cP_\lambda)$ is a sum of random variables determined by $\cP_\lambda \cap  (\cRbulk \cup \cRint)$, and therefore ${\tcLbulk}\ind_{F_\lambda}$ is determined by $\cP_\lambda \cap  (\cRbulk \cup \cRint)$ as claimed. 

Finally, it is clear from the geometry of the cone that $N_{\theta,\phi}(\by,\cP_{\lambda}) \in [0,1]^2 \cap \cCtp(\by) \subseteq \cRbndry$ for all $\by \in \cP_{\lambda} \cap \cRbndry$, so ${\tcLbndry}$ is determined by $\cP_{\lambda} \cap \cRbndry$.
 \end{proof}

Now we have all the ingredients to
prove part~\ref{thm:generalresult-singly-aligned} of our main result Theorem~\ref{thm:generalresult}.

\begin{proof}[Proof of Theorem~\ref{thm:generalresult}\ref{thm:generalresult-singly-aligned}]
Let $Z \sim \cN(0,1)$. 
Theorem~\ref{thm:bulk-clt}
shows that $\tcLbulk \tod s_{\phi}Z$ as $\lambda \to \infty$ for some $s_{\phi} \in (0, \infty)$.
Also, ${\tcLbndry} \tod  Q_{1}$ as $\lambda \to \infty$ by Theorem~\ref{thm:boundary-effects} and 
$\tcLint \toP 0$
 by Theorem~\ref{thm:proofa}. 
We show that ${\tcLbulk}\ind_{F^\rc_\lambda} \toP 0$ as $\lambda \to \infty$. To see this, for all $\eps > 0$, we have
\begin{equation}\label{prob-0}
    \Pr(|{\tcLbulk}\ind_{F_\lambda^\rc}|>\eps) \leq \Pr(F_\lambda^\rc) \to 0, \text{ as } \lambda \to \infty. 
\end{equation}
By Slutsky's theorem we have that $ \text{as}\ \lambda \to \infty$, ${\tcLbulk}\ind_{F_\lambda} = {\tcLbulk} - {\tcLbulk}\ind_{F^\rc_\lambda} \tod  s_{\phi}Z$. By the independence from Lemma~\ref{ll}, we have
\begin{equation}
{\tcLbulk}\ind_{F_\lambda}+ {\tcLbndry} \tod  s_{\phi}Z+Q_{1}, \ \text{as}\ \lambda \to \infty \label{eqm},
\end{equation}
and the random variables $Z$ and $Q_{1}$ are independent. 
Therefore, 
\begin{align}
    \tcL_\lambda = ({\tcLbulk}\ind_{F_\lambda}+{\tcLbndry}) + ({\tcLbulk}\ind_{F_\lambda^\rc} + {\tcLint})\tod  s_{\phi}Z+Q_{1} \nonumber,
\end{align}
by Slutsky's theorem and the fact that the first bracket converges in distribution to the random variable $s_{\phi}Z+Q_{1}$ by~\eqref{eqm}, 
and the second bracket converges in probability to $0\ \text{as}\ \lambda \to \infty$ by~\eqref{prob-0} and Theorem~\ref{thm:proofa}.
\end{proof}

\section*{Acknowledgements}
\addcontentsline{toc}{section}{Acknowledgements}

Much of this work was carried out when 
QT was a PhD student at Durham University, supported by
the Cultural Bureau, Embassy of Saudi Arabia (SACB) in London, and Jazan University (KSA). 
ChatGPT helped with the proof of Lemma~\ref{lem:nadim} and with proof-reading earlier drafts.
AW was supported by EPSRC grant EP/W00657X/1.
Some of this work was carried
 out during the programme ``Stochastic systems for anomalous diffusion'' (July--December 2024) hosted by the  Isaac Newton Institute, under EPSRC grant EP/Z000580/1.

\printbibliography

\end{document}